\documentclass[12pt]{article}
\usepackage{scalefnt}
\usepackage{amssymb}
\usepackage{amsfonts}
\usepackage{amsmath}
\usepackage[nohead]{geometry}
\usepackage[singlespacing]{setspace}
\usepackage[bottom]{footmisc}
\usepackage{indentfirst}
\usepackage{endnotes}
\usepackage[round]{natbib}
\usepackage{amscd}
\usepackage{float}
\usepackage[dvips]{graphicx}
\usepackage{epic,eepic, bez123, ebezier, pgf, rotating, xcolor}
\usepackage{subfigure}

\newtheorem{theorem}{Theorem}

\newtheorem{corollary}[theorem]{Corollary}

\newtheorem{definition}{Definition}

\newtheorem{lemma}{Lemma}

\newtheorem{proposition}{Proposition}

\newenvironment{proof}[1][Proof]{\noindent\textbf{#1.} }{\ \rule{0.5em}{0.5em}}
\makeatother
\input{tcilatex}
\begin{document}

\begin{center}
\textbf{{\Large Deterministic Equations for \\[0pt]
\textbf{Stochastic }Spatial Evolutionary Games}}\footnote{%
We thank William Sandholm for his useful suggestions and comments. The
research of the S.-H. H. was supported under grants NSF-DMS-0715125. The
research of M.A.K. was supported by the National Science Foundation through
the grant NSF-DMS-0715125, and the European Commission Marie-Curie grant
FP6-517911. The research of L. R.-B. was partially supported by the NSF
under grants NSF-DMS-06058.}

{\large \vspace{2mm} } {\large \vspace{2mm} }

{\large \textbf{Sung-Ha Hwang\footnote{%
corresponding author. email: hwang@math.umass.edu}} }

{\large \vspace{2mm} }

{\large \ {\small \textit{Department of Mathematics and Statistics,
University of Massachusetts, U.S.A. \\[0pt]
}} }

{\large \vspace{5mm} }

{\large \textbf{Markos Katsoulakis\footnote{%
email: markos@math.umass.edu}} }

{\large \vspace{2mm} }

{\large {\small \textit{Department of Mathematics and Statistics, University
of Massachusetts, U.S.A. \\[0pt]
Department of Applied Mathematics, University of Crete and Foundation of
Research and Technology-Hellas, Greece }} }

{\large \vspace{5mm} }

{\large \textbf{Luc Rey-Bellet\footnote{%
email: luc@math.umass.edu}} }

{\large \vspace{2mm} }

{\large {\small \textit{Department of Mathematics and Statistics, University
of Massachusetts, U.S.A. \\[0pt]
}} }
\end{center}

\begin{abstract}
Spatial evolutionary games model individuals who are distributed in a
spatial domain and update their strategies upon playing a normal form game
with their neighbors. We derive integro-differential equations as
deterministic approximations of the microscopic updating stochastic
processes. This generalizes the known mean-field ordinary differential
equations and provide a powerful tool to investigate the spatial effects in
populations evolution. The deterministic equations allow to identify many
interesting features of the evolution of strategy profiles in a population,
such as standing and traveling waves, and pattern formation, especially in
replicator-type evolutions.
\end{abstract}


\thispagestyle{empty} \onehalfspacing

\newpage

\section{Introduction}

Many macroeconomic phenomena occur as the aggregate results of the actions
and interactions of many and unrelated agents. Interactions among agents are
inherently local because they are separated by spatial location, language,
culture, and etc. As an example the occurrence of depression or inflation
might depend on the decentralized decisions of many agents to save or
consume based on the local economic conditions. \ Another example of a social
phenomenon where spatial considerations matter is the decision on where
to live, which frequently depends on the neighbors living there and this, in
turn, can induce the spatial patterns of segregation in residential areas %
\citep{Schelling71}. Our main goal in this paper is to develop tools to
understand the implications of spatial interactions in evolutionary games.

We consider a class of spatial stochastic processes in which agents are
located on vertices of a graph and update their strategy after observing the
strategy of a distinguished set of their neighbors. Such stochastic models
have been studied in evolutionary game theory by, among others, %
\citet{Kandori93}, \citet{Ellison93}, \citet{Blume93}, \citet{Blume95}, %
\citet{Young98}. In \emph{mean-field} models without any consideration of
geometrical proximity (e.g., see \citet{Kandori93}), every player is
considered a neighbor and is given the same weight in evaluating payoffs and
the model can be described using only the aggregate quantities such as the
proportion of players with a given strategy.

The popular use of mean-field stochastic models in evolutionary game
theory is mainly due to the simple structure of the stochastic processes.
While these models demonstrate effectively how the combination of the myopic
strategy revision and the random experiments (noise or mutation) of
individuals can lead to a selection of an equilibrium from among multiple
Nash equilibria, they assume that everyone in a population interacts \emph{%
uniformly} with the entire population of players, neglecting the local
structure of interaction. Because of this, these models fail to address the
importance of the locality of interactions in forming the globally observed
phenomenon. The range of interaction itself is a critical factor in
determining the speed of convergence to equilibrium and the time that the
society locked in a \textquotedblleft bad\textquotedblright\ state such as 
\textit{Pareto-inefficient} state. For instance, \citet{Ellison93} shows
that the speed of convergence in the system with the uniform interaction in
a large population is very slow and so the question of equilibrium selection
in the mean field models may not yield a useful insight.

By contrast, in models with local interactions players interact with a fixed
finite set of neighbors (e.g., see \citet{Blume93}, \citet{Blume95}, %
\citet[][Chapter 5]{Young98} and see \cite{Szabo07} for a comprehensive
survey of spatial evolutionary games). Some important questions about
equilibrium selection and speed of convergence to equilibrium have been
addressed for special models using tools such as stochastic stability and
Gibbs measures (\cite{Young93},\citet{Ellison93},\citet{Blume93}, \cite{Young98}, See also %
\citet{Freidlin84}). These results are however limited to either potential
games and logit dynamics or coordinations games and perturbed best response
dynamics; many other important games addressing economic problems $-~$such
as the Rock-paper-scissors games $-$ are neither potential games nor
coordination games. Also another important behavioral rule like imitative
updating, which may arise from the limited information or poor computational
ability of agents, might exhibit very different phenomena and should be
compared to the (perturbed) best response rule (See \citet{Levine07,
Bergin09} for the importance of imitative behaviors). However, these
questions using stochastic and probabilistic methods lead to very difficult
problems (e.g. see \cite{Durrett99}).

We concentrate here on an intermediate case, \emph{local mean-field models}
where a given player interacts with a substantial proportion of the
population, but where spatial variations in the strength of interactions are
nonetheless allowed. By considering the intermediate model, we are able to
rigorously derive deterministic equations which approximate the original
stochastic processes. The distinctive advantages of our approaches lie in
that under our assumptions, the questions of pattern formations,
equilibrium selection, and the speed of convergence at the level of complex
stochastic processes can be directly translated into those at the level of the
differential equations - a level which is more tractable and also incorporates 
some underlying stochastic details. An attractive
feature of the derived differential equations is that it can provide
tractable and systematic tools which can examine the relationship between the
local heterogeneity and the globally observed phenomena; for example, one
can easily study how a given set of preferences of individuals at a certain
neighborhood can lead to segregation patterns in the residential areas.

In the current literature of evolutionary game theory, (e.g. %
\citet{Hofbauer03,Weibull95, Sandholm08}), the time evolution of the
proportion of agents with strategy $i$ at time $t$, $f_{t}(i)$, is specified
by an ordinary differential equation of the type 
\begin{equation}
\frac{\partial }{\partial t}f_{t}(i)=\tsum\limits_{k\in S}\mathbf{c}%
^{M}(k,i,f_{t})f_{t}(k)-f_{t}(i)\tsum\limits_{k\in S}\mathbf{c}%
^{M}(i,k,f_{t})\,\,\,\text{ for }i\in S.  \label{intro-ODE}
\end{equation}%
Examples of such equations include the well-known replicator dynamics, the
von-Neumann Nash dynamics, and the logit dynamics. The first term in
equation (\ref{intro-ODE}) 
%$\tsum\nolimits_{k% \in S}\mathbf{c}^{M}\mathbf{(}k,i,f_{t})f_{t}(k)$
describes the rate at which agents switch to strategy $i$ from some strategy
other than $i$, while the second term 
%$~f_{t}(i)\tsum\nolimits_{k\in S}\mathbf{c}^{M}(i,k,f_{t})$
describes the rate at which agents switch to some other strategy from
strategy $i$. For this reason \ equation ($\ref{intro-ODE})~$ is also called
an \emph{input-output} equation.

It is well-known \citep{Kurtz70, Benaim03, Darling08} that the solution of
equation (\ref{intro-ODE}) $f_{t}(i)$ approximates, on finite time
intervals, a suitable mean-field stochastic process, in the limit of
infinite population, and $f_{t}(i)$ is the average, over the entire spatial
domain, of the proportion of player with strategy $i$. In our local
mean-field model where spatial structures survive, we will describe instead
the state of the system by a local density function $f_{t}(u,i)$. Here $u$
belongs to the spatial domain $\mathbb{A\subset }$~$\mathbb{R}^{n}$ where
agents are continuously distributed and $f_{t}(u,i)$ represents the
proportion of agents with strategy $i$ at the spatial location $u$. Our main
result is that local mean-field stochastic processes are approximated, on
finite time intervals and in the limit of infinite population, by equations
of the type%
\begin{equation}
\frac{\partial }{\partial t}f_{t}(u,i)=\tsum\limits_{k\in S}\mathbf{c}%
(u,k,i,f_{t})f_{t}(u,k)-f_{t}(u,i)\tsum\limits_{k\in S}\mathbf{c}%
(u,i,k,f_{t})\text{ \ for }i\in S\,,  \label{intro-IDE}
\end{equation}%
which provides a natural generalization of equation (\ref{intro-ODE}). For
example, the term $\mathbf{c(}u,k,i,f)$ describes the rate at which agents
at spatial location $u$ switch from strategy $k$ to $i$. This rate depends
on the strategies of agents at other spatial locations and typically, $%
\mathbf{c}(u,k,i,f)$ will have the functional form 
\begin{equation*}
\mathbf{c}(u,k,i,f)=G(k,i,\mathcal{J}\ast f(u,i)),\text{ where }\mathcal{J}%
\ast f(u,i):=\int \mathcal{J(}u-v)f(v,i)dv\,.
\end{equation*}%
Here $\mathcal{J}\ast f$ is the convolution product of $\mathcal{J}$ with $f$
and $\mathcal{J(}u)$ is a non-negative probability kernel which describes
the interaction strength between players whose relative distance is $u$.
When $\mathcal{J}$ is constant equation (\ref{intro-IDE}) reduces to
equation (\ref{intro-ODE}). Note that the rate of increases in $f_{t}$ at $u$
depends on $f_{t}(v,i)$ for all $v$ in the spatial domain $\mathbb{A}$ and
that equation (\ref{intro-IDE}) is non-local; so, the equation is called an 
\emph{integro-differential equation} (IDE).

We will obtain this equation in a suitable \emph{spatial} scaling where
agents are continuously distributed in a spatial domain $\mathbb{A}$ and is
often called a \emph{mesoscopic scaling} in the physics literature. For this
reason equations of the form (\ref{intro-IDE}) are often called mesoscopic
equations but are sometimes also referred to as local mean-field equations.
Mesoscopic limits similar to ours have been derived in several models in
statistical mechanics by \citet{Comets87}, \citet{DeMasi94}, %
\citet{Katsoulakis05}, and others. We generalize these results to the
spatial stochastic processes arising in evolutionary game theory. Other
scaling limits, such as hydrodynamic limits, where space and time are scaled
simultaneously, giving rise to evolving fronts and interfaces, e.g. \cite%
{Katsoulakis97}, are also potentially relevant to game theory but will not
be discussed here.

\begin{figure}[tb]
\centering
\subfigure[Traveling
front]{\label{fig:traveling}\includegraphics[width=0.3\textwidth]{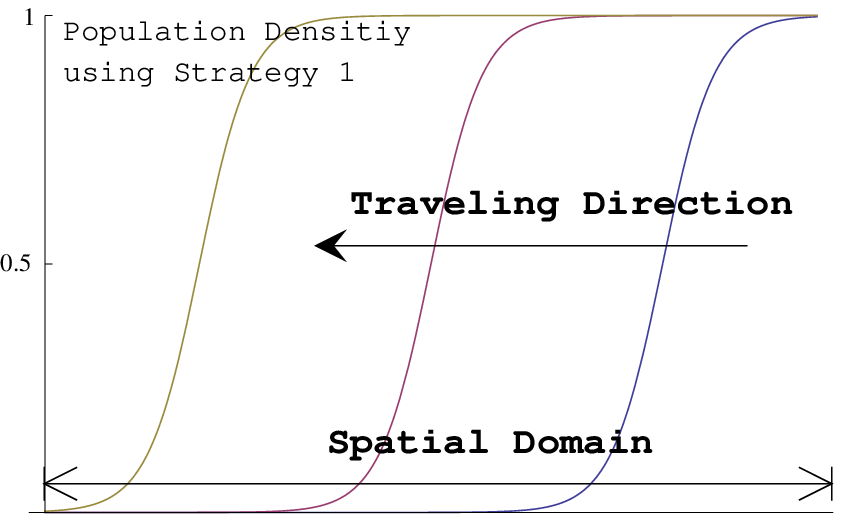}}
\qquad 
\subfigure[Imitative update vs
perturbed best
response]{\label{fig:schematic}\includegraphics[width=0.5\textwidth]{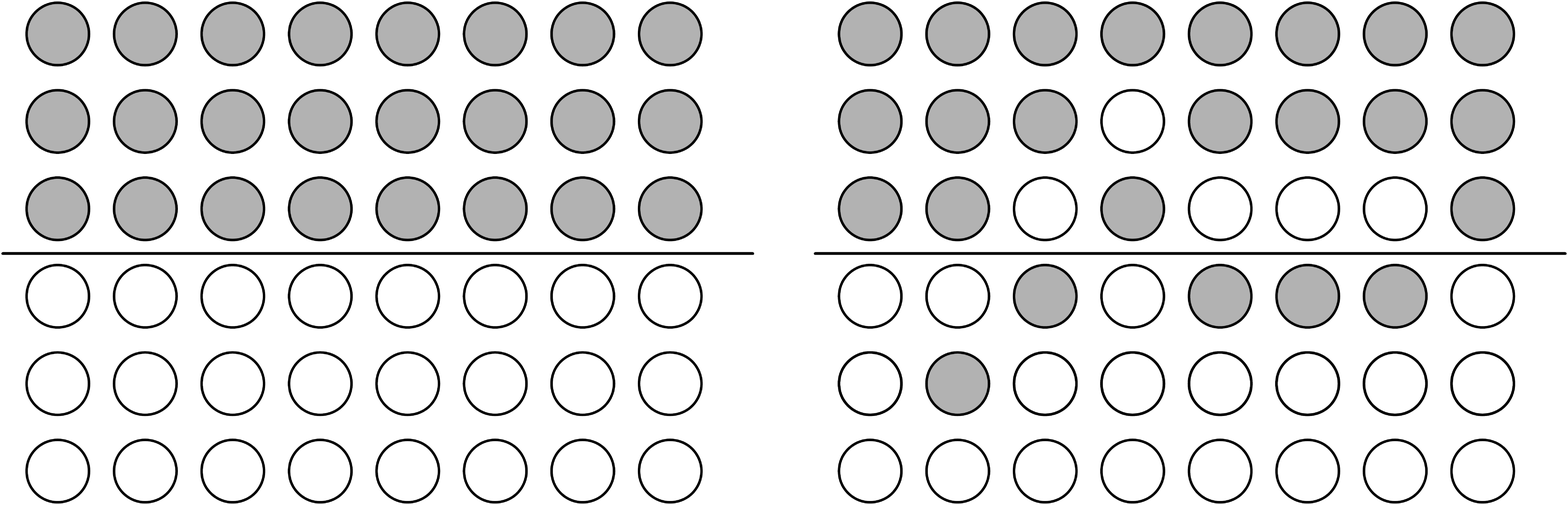}}
\caption{\textbf{Traveling front and strategy choices} 
{\protect\footnotesize {Panel (a) illustrates how a traveling front solution
can describe the propagation of a strategy to the whole spatial domain for a
two strategy game. Panel (b) shows the configurations of strategy choices of
individuals at each site (white circle: strategy 1, black circle: strategy
2) In the replicator dynamics (the left: an example of imitative updating
rules) \ the interface between the choices of strategies is sharp; in the
logit dynamics (the right: an example of perturbed best response rules) the
interface is not sharp }}}
\label{fig:travel-front}
\end{figure}

Specifically, using this approximated equation we study the effect of a
given initial condition and behavioral update rule on pattern formations and
observe that when individuals behave by imitation of their close neighbors,
the segregation of choices of strategy may develop and persist (see Figure %
\ref{fig-pattern}). But in a society where (perturbed) best response rule is
the dominating behavior, we observe that the system, instead, converges
everywhere to a \textquotedblleft rational\textquotedblright\ equilibrium
exponentially fast (see Figure \ref{fig:travel-front}(b) for comparison of
interfaces). The traveling front solutions shows a propagation of a local
strategy profile into a global domain of the space and plays an important
role in examining the selection of equilibrium from among multiple Nash
equilibria and the speed of convergence to equilibrium (see Figure \ref%
{fig:travel-front} (a), \citet{Hofbauer97}). In our models, in imitative
dynamics, we observe an extremely slow transition to the better equilibrium
in contrast to standard beliefs on equilibrium selection. 
\begin{figure}[tb]
\centering 
\includegraphics[scale=0.6]{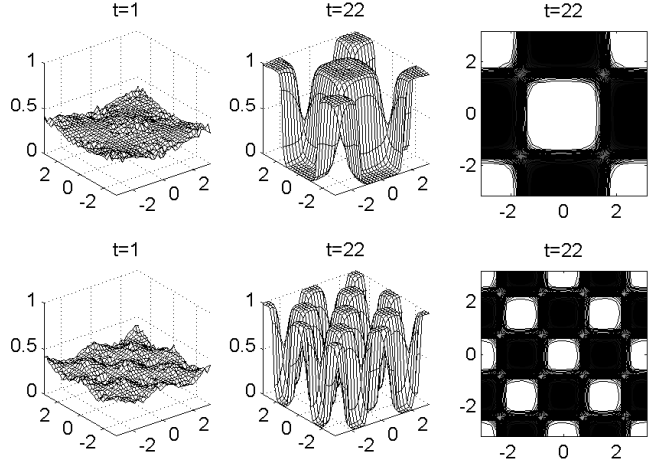}
\caption{\textbf{Pattern Formation in the replicator dynamics } 
{\protect\footnotesize {\ We use two-player coordination game with $%
a_{11}=2/3,a_{22}=1/3,$\ $a_{12}=a_{21}=0$. The left and middle panels show
the time evolutions of population densities of strategy 1 in the spatial
domain $T^{d}=[-\protect\pi ,\protect\pi ]^{2}.$\ The number of nodes is $64$%
\ and the time step is $0.0175.$ The initial conditions are $1/3+\func{rand}%
\cos (x)\cos (y)\ $(upper panel) and $1/3+\func{rand}\cos (2x)\cos (2y)\ $%
(lower panel), where rand denotes a realization of the uniform random
variable $[0,1]$\ at each node. \ For the interaction kernel, }}$\mathcal{J}$%
{\protect\footnotesize {$(r)=\exp \left( -bx^{2}\right) /\protect\int \exp
(-bx^{2})dx,b=15.$\ The right panels show the contours of the densities at $%
t=22.$ }}}
\label{fig-pattern}
\end{figure}

Current approaches to address the pattern formation and the existence of
traveling front solutions in evolutionary games have traditionally employed
reaction-diffusion partial differential equations; such models are typically
obtained by adding a constant coefficient diffusion term to the mean-field
equations, which in turn models fast but homogeneous spatial diffusion of
agents\citep{Hutson92, Vickers93, Hofbauer97, HofbauerWave97, Durrett99}:%
\begin{equation}
\frac{\partial }{\partial t}f_{t}(u,i)=\tsum\limits_{k\in S}\mathbf{c}%
^{M}(k,i,f)f_{t}(u,k)-f_{t}(u,i)\tsum\limits_{k\in S}\mathbf{c}%
^{M}(i,k,f)+\epsilon \Delta f.  \label{intro-PDE}
\end{equation}%
In contrast, in our scaling-limit approach, the spatial effects are
introduced at the microscopic level and lead to diffusive effects which
differ markedly from equation to equation and are, in general, density
dependent. This introduces a number of new interesting spatial structures
which are absent in reaction-diffusion equations.

The paper is organized as follows. In Section 2 we introduce the stochastic
process and the scaling limits, and present our main results (Section 2.3).
A heuristic derivation of the equation is given in Section 2.4 and the
relation with ODE such as (\ref{intro-ODE}) is elucidated in Section 2.5. In
Section 3 we analyze equilibrium selection and pattern formation in two
player coordination games using a combination of linear analysis and
numerical simulations. In the appendix we prove our main results.

%\begin{figure}[b]
%\centering 
%\includegraphics[scale=0.15]{figure4.jpg} 
%\caption{\textbf{Probabilities of switching in a spatial domain.}
%\footnotesize{We suppose that interactions
%arise among the nearest neighbors. The agent at site A faces five
%1-strategists (white circle) out of 8 neighbors, while the agent at site B
%interacts with only three 1-strategists. Thus agent A may switch his
%strategy with probability $\frac{5}{8}$; agent B does this with probability $%
%\frac{3}{8}.$}}
%\label{fig:schematic2}  
%\end{figure}

\section{Spatial Evolutionary Games}

\subsection{Strategy Revision Processes\label{model-section}}

In models of spatial evolutionary games, agents are located at the sites of
a graph and play a normal form game with their neighbors. The graph $\Lambda 
$ is assumed here to be a subset of the integer lattice $\mathbb{Z}^{d}$. We
consider a single population playing a normal form game, but the
generalization to multiple population games is straightforward. A normal
form game is specified by a finite set of strategies $S$ and a payoff
function $a(i,j)$ which gives the payoff for a player using strategy $i\in S$
against strategy $j\in S$. Here we view a strategy as a type of behavior and
so terms, \textquotedblleft strategy\textquotedblright\ and
\textquotedblleft type\textquotedblright , are used interchangeably %
\citep{Smith82, Hofbauer98}.

The strategy of the agent at site $x\in \Lambda $ is $\sigma _{\Lambda
}\left( x\right) \in S$, and we denote\ by $\sigma _{\Lambda }=\left\{
\sigma _{\Lambda }(x):x\in \Lambda \right\} $ the configuration of
strategies for every agent in the population. With these notations, the
state space, i.e., the set of all possible configurations, is $S^{\Lambda }.$
The subscript of $\sigma =\sigma _{\Lambda }$ will be suppressed, whenever
no confusion arises. As in \citet[][chapter 6]{Young98}, we assign positive
weights $\mathcal{W}(x-y)$ to any two sites $x$ and $y$ to capture the
importance or intensity of the interaction among neighbors. Note that we
assume that these weights depend only on the relative location $x-y$ between
the players (i.e., we assume translation invariance). It is convenient to
assume that total weight that site $x$ attaches to all its neighbors is
normalized to 1, i.e., 
\begin{equation}
\sum_{y\in \Lambda }\mathcal{W}(x-y)\approx 1.  \label{eq:w-normal}
\end{equation}%
We say $y$ is a neighbor of $x$ whenever $\mathcal{W}(x-y)>0$. An individual
agent, at site $x$ with strategy $i$ given a configuration $\sigma $,
receives an average payoff 
\begin{equation}
u(x,\sigma ,i):=\sum_{y\in \Lambda }{\mathcal{W}}(x-y)a(i,\sigma (y)).
\label{avpayoff}
\end{equation}%
If we think of $\mathcal{W}$ as the probability with which an agent samples
his neighbors, then $u(x,\sigma ,i)$ is the expected payoff for an agent at $%
x$ choosing strategy $i$ if the population strategy profile is $\sigma $. Or
we can think that an agent receives an instantaneous payoff flow from his
interactions with other neighbors \citep{Blume93, Young98, Young01}.

In the special case where $\mathcal{W}(x-y)\ $is constant, the interaction
is uniform and there is no spatial structure and if there are a total of $%
n^{d}$ agents in the population, then $\mathcal{W}(x-y)\approx \frac{1}{n^{d}%
}$ because of (\ref{eq:w-normal}). On the other hand, when $\mathcal{W}(x-y)=%
\frac{1}{2d}$ if $\left\Vert x-y\right\Vert =1$ and $0$ otherwise,
interactions only arise between nearest sites \citep{Blume95, Szabo07}.

In this paper we concentrate on long range interactions where each agent
interacts with as many other agents as in the mean-field case, but the
interaction is not uniform. This limit is known as \textquotedblleft local
mean field model\textquotedblright\ \citep{Comets87} or \textquotedblleft
Kac potential\textquotedblright\ \citep{Lebowitz66, DeMasi94, Presutti09}.
More specifically, let $\mathcal{J}(x)$ be a non-negative, compactly
supported,$\ $and integrable function such that $\int \mathcal{J}(x)dx=1.$
We assume that $\mathcal{W}$ has the form:%
\begin{equation}
\mathcal{W}^{\gamma }(x-y)=\gamma ^{d}\mathcal{J}(\gamma (x-y))\,,
\label{Kac-pot}
\end{equation}%
and we will take the limit $\Lambda \nearrow \mathbb{Z}^{d}$ and $\gamma
\rightarrow 0$ in such a way that $\gamma ^{-d}\approx \left\vert \Lambda
\right\vert \approx n^{d}.$ Here $n^{d}$ is the size of the population and $%
\left\vert \text{ }\right\vert $ denotes the cardinality. Hence the factor $%
\gamma ^{d}$ is chosen in such a way that $\sum \mathcal{W}^{\gamma
}(x-y)\approx \int \mathcal{J}(x)dx=1,$ so $\mathcal{W}^{\gamma }(x-y)$
indeed represents the intensity of interactions. Note that in (\ref{Kac-pot}%
) the interaction vanishes when $\left\Vert x-y\right\Vert \geq R\gamma
^{-1} $ if $\mathcal{J}$ is supported on the ball of radius $R$. So as $%
\gamma \rightarrow 0$, an agent interacts very weakly but with a growing
number of neighbors in the population. Frequently, in examples and
simulations, we consider localized Gaussian-like kernels $\mathcal{J}%
(x)\propto \exp (-b\left\Vert x\right\Vert ^{2})$\ for some $b>0$. 
%in  which case interactions among nearer neighbors receive more weight than ones among
%further neighbors.\ 

The time evolution of the system is given by a continuous time Markov
process $\{\sigma _{t}\}$ with state space $S^{\Lambda }$, in which each
agent receives, independently of all the other agents, a strategy revision
opportunity in response to his own exponential $``$alarm
clock\textquotedblright\ with rate $1$, and then updates his strategy
according to a rate $c(x,\sigma ,k)$ $-$ the rate with which agent $x$
switches to strategy $k$ when the configuration is $\sigma $. This process
is then characterized by a generator %\citep{Ligget85}:%
\begin{equation}
\left( Lg\right) (\sigma )=\sum_{x\in \Lambda }\sum_{k\in S}c(x,\sigma
,k)\left( g(\sigma ^{x,k})-g(\sigma )\right)  \label{gen-sigma}
\end{equation}%
where $g$ is a bounded function on $S^{\Lambda }$ and 
\begin{equation*}
\sigma ^{x,k}(y)=\left\{ 
\begin{tabular}{ll}
$\sigma (y)$ & if $y\neq x$ \\ 
$k$ & if $y=x$%
\end{tabular}%
\right.
\end{equation*}%
represents a configuration where the agent at site $x$ switches from his
current strategy $\sigma (x)$ to a new strategy $k$.

If $c\left( x,\sigma ,k\right) >0$ for all $x$, $\sigma $ and $k$, the
stochastic process can introduce any new strategy, even if it is not
currently used in the population. We call this case \emph{innovative}
following \citet{Szabo07}. If $c(x,\sigma ,k)=0$ for some $x,$ $\sigma ,$
and $k$ and hence a strategy which is not present in the population does not
appear under the dynamics, we call the dynamics \emph{non-innovative}.
Furthermore if, upon switching, agents only consider the payoff of the new
strategy we call the dynamics \emph{targeting}. In contrast, when agents'
decision depends on the payoff difference between the current strategy and
the new strategy we call the dynamics \emph{comparing}.

Precise technical assumptions for the strategy revision rates will be
discussed later (Conditions \textbf{C1}$-$\textbf{C3 } in Section \ref%
{main-theorem}); here, we give only a number of concrete examples commonly
used in applications, and to which our results will apply. Several more
examples of rates are discussed in the Appendix and the assumptions on our
rates are satisfied by virtually all dynamics commonly used in evolutionary
game theory 
\citep[see][for a
more comprehensive discussion of  rates and more examples]{Sandholm08}. 
%Let $\delta _{x}$ denote the Dirac delta measure at $x:$ i.e., $\delta _{x}\left(
%A\right) =1$ if $x\in A$ and $0$ otherwise, and $F$ be a non-negative
%Lipschitz continuous function. 
To define the rate we introduce 
\begin{equation*}
w(x,\sigma ,k):=\sum_{y\in \Lambda }\mathcal{W}(x-y)\delta (\sigma (y),k)
\end{equation*}%
where $\delta (i,j)=1$ if $i=j$ and $0$ otherwise; $w(x,\sigma ,k)$ can be
interpreted as the probability for an agent at site $x$ to find a neighbor
with strategy $k$, provided the neighbors are sampled with the probability
distribution $\mathcal{W}(x-y)$. %\begin{equation*}
%w(x,y,\sigma ,k):=\mathcal{W}(x-y)\delta _{\sigma (y)}\left( \left\{
%k\right\} \right) ,
%\end{equation*}%
%which is the probability of an agent at site $x$ meeting $y$ among $%
%k- $strategists in his neighborhood: the weight $x$ attaches to his neighbor
%of type $k$ at configuration $\sigma .$\newline
Let also $F$ denote a non-negative function. \newline

\noindent \underline{\textbf{Examples of Rates}}

\medskip \noindent $\bullet$ \textbf{Comparing and Innovative}: The rate is $%
c(x,\sigma ,k)=F(u(x,\sigma ,k)-u(x,\sigma ,\sigma (x)))$ and is comparing
and innovative provided $F>0$. When 
\begin{equation*}
c(x,\sigma ,k)=\min \left\{ 1,\exp \left( u(x,\sigma ,k\right) -u(x,\sigma
,\sigma (x)))\right\} ,
\end{equation*}
the rate corresponds to a generalization of the well-known Metropolis
algorithm. In particular, when the normal form game is a potential game, the
corresponding Markov chain satisfies detailed balance and its invariant
distribution can be explicitly expressed as a Gibbs distribution %
\citep{Szabo07}.

\medskip \noindent $\bullet$ \textbf{Targeting and Innovative}: This case
arises if $c(x,\sigma ,k)=F(u(x,\sigma ,k))$ and $F>0.$ If 
\begin{equation}
c(x,\sigma ,k)=\frac{\exp (u(x,\sigma ,k))}{\sum_{l}\exp (u(x,\sigma ,l))}
\label{logit}
\end{equation}%
the rate is called \textquotedblleft logit choice rule\textquotedblright\ in
the game theory literature, and it is a generalization of the
\textquotedblleft Gibbs sampler\textquotedblright\ in statistics and of the
\textquotedblleft Glauber dynamics \textquotedblright of physics. The Markov
process, in this case too, satisfies the detailed balance for potential
games and has the same Gibbs invariant distribution as Metropolis dynamics.

\medskip \noindent $\bullet$ \textbf{Comparing and Non-innovative: }The rate
has the form 
\begin{equation}
c(x,\sigma ,k)=w(x,\sigma ,k)\left[ F(u(x,\sigma ,k)-u(x,\sigma ,\sigma (x))%
\right] \,.  \label{rep-imitative}
\end{equation}%
This specifies an imitation process: the first factor $w(x,\sigma ,k)$ is
the probability for an agent at $x$ to choose an agent with strategy $k$ and
the second factor $F(u(x,\sigma ,k)-u(x,\sigma ,\sigma (x)))$ gives the rate
at which the new strategy $k$ is adopted 
\citep{Weibull95, Benaim03,
Hofbauer03}. The standard example is 
\begin{equation}
c(x,\sigma ,k)=w(x,\sigma ,k)\left[ u(x,\sigma ,k)-u(x,\sigma ,\sigma (x))%
\right] _{+}  \label{rate-rep1}
\end{equation}%
where $\left[ s\right] _{+}=\max \left\{ s,0\right\} $. The rate (\ref%
{rate-rep1}), in the mean-field case, gives rise to the famous replicator
ODEs as the deterministic approximation. More generally if $F$ in (\ref%
{rep-imitative}) satisfies %\begin{equation}
%c(x,\sigma ,k)= w(x,\sigma ,k) \left[ F(u(x,\sigma
%,k)-u(x,\sigma ,\sigma (x)) \right] \label{rate-rep2}
%\end{equation}%
\begin{equation}
F(s)-F(-s)=s\,,  \label{con_rep}
\end{equation}%
%
%
%
%
%
%
%
%
%
%
%
%
%
%
%
%
%
%
%\begin{equation}
%F(u(x,\sigma ,k)-u(x,\sigma ,i))-F(u(x,\sigma ,i)-u(x,\sigma ,k))=u(x,\sigma
%,k)-u(x,\sigma ,i)  \label{con_rep}
%\end{equation}%
then the corresponding mean field ODE is the replicator dynamics. Note that $%
\left[ s\right] _{+}$ satisfies condition (\ref{con_rep}). In the paper we
frequently use 
\begin{equation}
F_{\kappa }\left( s\right) :=\frac{1}{\kappa }\log (\exp (\kappa s)+1)
\label{reg_rep}
\end{equation}%
and it is easily seen that the function (\ref{reg_rep}) satisfies (\ref%
{con_rep}) and converges uniformly to $\left[ s\right] _{+}$ as $\kappa
\rightarrow \infty ;$ hence (\ref{reg_rep}) can serve as a smooth
regularization of (\ref{rate-rep1}), and we name a replicator equation using
(\ref{reg_rep}) by a regularized replicator equation.

\subsection{Mesoscopic scaling and long-range interactions\label%
{domain-section}}

We will consider the limit $\gamma \rightarrow 0$ in equation (\ref{Kac-pot}%
); i.e., the interaction range $\frac{1}{\gamma }$ becomes infinite and the
agent at $x$ interacts with a growing number of agents. In order to obtain a
limiting equation, we rescale space and take a continuum limit. Let $\mathbb{%
A\subset R}^{d}$ (mesoscopic domain) and $\mathbb{A}^{\gamma }:=\mathbb{%
\gamma }^{-1}\mathbb{A\cap Z}^{d}$ (microscopic domain). If $\mathbb{A}$ is
a smooth region in $\mathbb{R}^{d}$, then $\mathbb{A}^{\gamma }$ contains $%
\gamma ^{-d}|\mathbb{A}|$ lattice sites and as $\gamma \rightarrow 0,$ $%
\gamma \mathbb{A}^{\gamma }$ approximates $\mathbb{A}$.

At the mesoscopic scale the state of the system is described by the \emph{%
strategy profile} function $f_{t}(u,i)$ -- the density of agents with
strategy $i$ at $u$. The bridge between microscopic and mesoscopic scale is
given by the \emph{empirical measure} $\pi ^{\gamma }(\sigma ;du,di)$
defined as follows. For $(v,j)\in \mathbb{A}\times S$, let $\delta
_{(v,j)}(du,di)$ denote the Dirac delta measure at $(v,j)$.

\begin{definition}[Empirical measure]
The empirical measure $\pi ^{\gamma }:S^{\mathbb{A}^{\gamma }}\rightarrow 
\mathcal{P}(\mathbb{A}\times S)$ is the map given by 
\begin{equation}
\sigma \mapsto \pi ^{\gamma }(\sigma ;du,di):=\frac{1}{\left\vert \mathbb{A}%
^{\gamma }\right\vert }\sum_{x\in \mathbb{A}^{\gamma }}\delta _{\left(
\gamma x,\sigma (x)\right) }(dudi)  \label{empirical}
\end{equation}%
where $\mathcal{P}(\mathbb{A}\times S)~$denotes the set of all probability
measures on $\mathbb{A}\times S.$
\end{definition}

Our main result is to show that, under suitable conditions, 
\begin{equation}
\pi ^{\gamma }(\sigma _{t};du,di)\rightarrow f_{t}(u,i)du\quad\quad \text{
in probability}  \label{res-conv}
\end{equation}%
and $f_{t}(u,i)$ satisfies an integro-differential equation. Since $\sigma
_{t}$ is the state of the microscopic system at time $t,$ $\pi ^{\gamma
}(\sigma _{t};du,di)$ is a random measure, while $f_{t}(u,i)$ is a solution
of a deterministic equation. So (\ref{res-conv}) is in a sense a form of a
time-dependent law of large numbers. For this result to hold we need to
assume that the initial distribution for $\sigma _{0}$ is sufficiently
regular. For our purpose it will be enough to assume that the distribution
of $\sigma _{0}$ is given by a \emph{product measure with a slowly varying
parameter}.

\begin{definition}[Product measures with a slowly varying parameter]
\textsl{\label{product-measures copy(1)} }The collection of measure $\{\mu
^{\gamma }\}$ is called a family of product measures with a slowly varying
parameter if $\mu _{\gamma }:=\bigotimes\nolimits_{x\in \mathbb{A}^{\gamma
}}\rho _{x}$ on $S^{\mathbb{A}^{\gamma }}$ and there exists a profile $%
f(u,i) $ such that 
\begin{equation*}
\rho _{x}\left( \left\{ i\right\} \right) =f(\gamma x,i)
\end{equation*}
\end{definition}

More general initial distributions can also be accommodated %
\citep[See][]{Kipnis99}, provided they can be associated to a mesoscopic
strategy profile. Furthermore, below we will consider two types of boundary
conditions:

\medskip

\noindent \textbf{(a) Periodic Boundary Conditions.} Let $\mathbb{A=[}%
0,1]^{d}$ . We assume that $\mathbb{A}^{\gamma }=\gamma ^{-1}\mathbb{A\cap Z}%
^{d}=[0,\frac{1}{\gamma }]^{d}\mathbb{\cap Z}^{d},$ and then extend the
profile $f_{t}(u,i)$ and the configuration $\sigma _{\mathbb{A}^{\gamma }}$
periodically on $\mathbb{R}^{d}$ and $\mathbb{Z}^{d}.$ Equivalently we can
identify $\mathbb{A}$ with the torus $\mathbf{T}^{d}$ and similarly $\mathbb{%
A}^{\gamma }$ with the discrete torus $\mathbf{T}^{d,\gamma }.$

\medskip

\noindent \textbf{(b) Fixed Boundary Conditions.} In applications it is also
useful to consider the case where the configurations in some regions do not
change with time. Let $\Lambda \subset \Gamma \subset \mathbb{R}^{d}$ be a
region. We think of $\partial \Lambda :=$ $\Gamma \backslash \Lambda $ as
the \textquotedblleft boundary region\textquotedblright\ where agents do not
revise their strategies. Since we consider compactly supported $\mathcal{J}$ 
% \ which is
%compactly supported in a closed ball $B(0,r)$ for some $r>0,$ without loss
%of generality 
we can take, for suitable $r>0$ 
\begin{equation*}
\Gamma :=\bigcup\limits_{u\in \Lambda }B(u,r),
\end{equation*}%
where $B$ denotes a ball centered at $u$ with radius $r.$ We define
microscopic spaces, $\Lambda ^{\gamma }:=\gamma ^{-1}\Lambda \cap \mathbb{Z}%
^{d}$ and $\Gamma ^{\gamma }:=\gamma ^{-1}\Gamma \cap \mathbb{Z}^{d}.$

\subsection{Main results\label{main-theorem}}

Let us consider first the case with periodic boundary conditions. %
% where  the mesoscopic domain is the $d$-dimensional torus $\mathbf{T}^{d}$ 
%and $\mathbf{T}^{d,\gamma }$ the corresponding microscopic discrete torus.  
Our assumptions on the interactions weights $\mathcal{W}^\gamma\left(
x-y\right)$ are

\medskip

\noindent \textbf{(F)} $\ \mathcal{W}^{\gamma }\left( x-y\right) =\gamma ^{d}%
\mathcal{J}\left( \gamma (x-y)\right) $ where $\mathcal{J}$ is nonnegative,
continuous with compact support, and normalized, $\int \mathcal{J(}x)dx=1$.

\medskip

\noindent The mesoscopic strategy profiles are described by functions $%
f(u,i)\in \mathcal{M(}\mathbf{T}^{d}\times S)$ where%
\begin{equation*}
\mathcal{M(}\mathbf{T}^{d}\times S):=\left\{ f\in L^{\infty }(\mathbf{T}%
^{d}\times S\mathbb{)}:0\leq f(u,i)\leq 1,\text{ }\sum\nolimits_{i}f(u,i)=1%
\text{ for all }u\in \mathbf{T}^{d}\right\} \,.
\end{equation*}
Let $\left\{ \sigma _{t}^{\gamma}\right\} _{t\geq 0}$ be the stochastic
process with generator\textsl{\ }$L^{\gamma }$ given by 
\begin{equation}
(L^{\gamma }g)(\sigma )=\sum_{x\in \mathbf{T}^{d,\gamma }}\sum_{k\in
S}c^{\gamma }(x,\sigma ,k)\left( g(\sigma ^{x,k})-g(\sigma )\right)
\label{generator-per}
\end{equation}%
for $g$ \ $\in L^{\infty }(S^{\mathbf{T}^{d,\gamma }})$. 
% and initial distribution $\mu ^{\gamma }.$
Our assumptions on the strategy revision rate $c^{\gamma }(x,\sigma ,k)$ are
that there exists a real-valued function 
\begin{equation*}
\mathbf{c}(u,i,k,\pi ),\text{ }u\in \mathbf{T}^{d},\text{ }i,k\in S,\text{ }%
\pi \in \mathcal{P}(\mathbf{T}^{d}\times S)
\end{equation*}%
such that

\medskip \noindent \textbf{(C1)} $\mathbf{c}(u,i,k,\pi )$ satisfies%
\begin{equation*}
\lim_{\gamma \rightarrow 0}\sup_{x\in \mathbf{T}^{d,\gamma },\sigma \in S^{^{%
\mathbf{T}^{d,\gamma }}},k\in S}\left\vert c^{\gamma }(x,\sigma ,k)-\mathbf{c%
}(\gamma x,\sigma (x),k,\pi ^{\gamma }(\sigma ))\right\vert =0\,,
\end{equation*}

\medskip \noindent \textbf{(C2)} $\mathbf{c}(u,i,k,\pi )$ is\ uniformly
bounded: i.e., there exists $M$ such that%
\begin{equation*}
\sup_{u\in \mathbf{T}^{d},i,k\in S,\pi \in \mathcal{P(}\mathbf{T}^{d}\times
S)}\left\vert \mathbf{c}(u,i,k,\pi )\right\vert \leq M\,,
\end{equation*}

\medskip \noindent \textbf{(C3)} $\mathbf{c}(u,i,k,fdudi)$ satisfies a
Lipschitz condition with respect to $f$ : i.e., there exists $L$ such that
for all $f_{1},$ $f_{2}\in \mathcal{M}(\mathbf{T}^{d}\times S)$%
\begin{equation*}
\sup_{u\in \mathbf{T}^{d},i,k\in S}\left\vert \mathbf{c}(u,i,k,f_{1}dudi)-%
\mathbf{c}(u,i,k,f_{2}dudi)\right\vert \leq L\left\Vert
f_{1}-f_{2}\right\Vert _{L^{1}(\mathbf{T}^{d}\times S)}\text{ } .
\end{equation*}

\medskip

In the appendix we show that all the classes of rates given in the examples
in Section \ref{model-section} and several others satisfy conditions \textbf{%
C1}$-$\textbf{C3}. For example, if $c^{\gamma }(x,\sigma ,k)=F(u(x,\sigma
,k)-u(x,\sigma ,\sigma (x))),$ with $u(x,\sigma ,i)=\sum_{y}{\mathcal{%
W^{\gamma }}}(x-y)a(i,\sigma (y))$ and the weights $\mathcal{W^{\gamma }}%
(x-y)$ satisfy condition $\mathbf{F}$, then 
\begin{equation}
\mathbf{c}(u,i,k,f)=F\left( \tsum\limits_{l\in S}a(k,l)\mathcal{J}\ast
f(u,l)-a(i,l)\mathcal{J}\ast f(u,l)\right)  \label{conti-rate1}
\end{equation}%
satisfies condition $\mathbf{\ C1-C3}$ (recall that $\mathcal{J}\ast
f(u,l):=\int_{\mathbf{T}^{d}}\mathcal{J}(u-v)f(v,l)dv$ is the convolution of 
$\mathcal{J}$ with $f$). A slight modification of (\ref{conti-rate1}) yields
corresponding expressions for each choice of $c^{\gamma }(x,\sigma ,k)$ in
Section \ref{model-section} (see the appendix for the complete list of these
rates).\ In Section 2.4 we will explain how to obtain the function $\mathbf{c%
}(u,i,k,\pi )$ from the rates.

The stochastic process $\sigma _{t}^{\gamma }$ induces a measure-valued
stochastic process $\pi _{t}^{\gamma }:=\pi ^{\gamma }(\sigma _{t}^{\gamma
},dudi)$ for the empirical given in equation (\ref{empirical}). Theorem \ref%
{them-long-range-perBC} shows that the stochastic process $\pi _{t}^{\gamma
} $ has a deterministic limit.

\begin{theorem}[Long Range Interaction and Periodic Boundary Condition]
\label{them-long-range-perBC} Suppose the revision rate satisfies $\mathbf{C1%
}-\mathbf{C}3$. %and $\mathcal{J}\text{ satisfies }\mathbf{F}.$ 
Let $f\in \mathcal{M}(\mathbf{T}^{d}\times S)$ and assume that the initial
distribution $\left\{ \mu ^{\gamma }\right\} _{\gamma }$ is a family of
measures with a slowly varying parameter associated to the profile of $f$.
Then for every $T>0$%
\begin{equation*}
\lim_{\gamma \rightarrow 0}\pi _{t}^{\gamma }(du,di)=f_{t}(u,i)\,dudi\text{
in probability }
\end{equation*}%
\ uniformly for $t\in \lbrack 0,T]$ and $f_{t}$ satisfies the following
differential equation: for $u\in \mathbf{T}^{d},i\in S$%
\begin{eqnarray}
\frac{\partial }{\partial t}f_{t}(u,i) &=&\tsum\limits_{k\in S}\mathbf{c}%
(u,k,i,f)f_{t}(u,k)-f_{t}(u,i)\tsum\limits_{k\in S}\mathbf{c}(u,i,k,f)
\label{eq-de-per} \\
f_{0}(u,i) &=&f(u,i)  \notag
\end{eqnarray}
\end{theorem}

Next let us consider fixed boundary conditions as in Section \ref%
{domain-section}. In this case,\textsl{\ }the stochastic process, $\{\sigma
_{t}^{\gamma }\}_{t\geq 0}$, is specified by the generator $L^{\gamma }$%
\begin{equation}
(L^{\gamma }g)(\sigma _{\Gamma ^{\gamma }})=\sum_{x\in \Lambda ^{\gamma
}}\sum_{k\in S}c^{\gamma }(x,\sigma _{\Gamma ^{\gamma }},k)(g(\sigma
_{\Gamma ^{\gamma }}^{x,k})-g(\sigma _{\Gamma ^{\gamma }}))
\label{eq-generator}
\end{equation}%
for $g$ \ $\in L^{\infty }(S^{\Gamma ^{\gamma }}).$ Note that the summation
in terms of $x$ in (\ref{eq-generator}) is taken over $\Lambda ^{\gamma }$,
which represents the fact that only individuals in $\Lambda ^{\gamma }$
revise their strategies, whereas the rate depends on the configuration in
entire $\Gamma ^{\gamma }$. For a given $f\in \mathcal{M}$, we define its
restriction on $\Lambda $, $f_{\Lambda }(u,i):$ $f_{\Lambda }(u,i)=f(u,i)$
if $u\in \Lambda $ and $f_{\Lambda }(u,i)=$ 0 if $u\in \Lambda ^{C}$. 
% prove, by similar arguments as in theorem \ref%
%{thm-long-range-fixedBC}, the following theorem:

\begin{theorem}[Long Range Interaction and Fixed Boundary Condition]
\label{thm-long-range-fixedBC} Suppose the revision rate satisfies $\mathbf{%
C1}-\mathbf{C3}$. % and $\mathcal{J}\text{ satisfies }\mathbf{F}.$ 
Let $f\in \mathcal{M}(\Gamma ^{d}\times S)$ and assume that the initial
distribution $\left\{ \mu ^{\gamma }\right\} _{\gamma }$ is a family of
measures with a slowly varying parameter associated to the profile of $f.$
Then for every $T>0$%
\begin{equation*}
\lim_{\gamma \rightarrow 0}\pi _{t}^{\gamma }(du,di)=\frac{1}{\left\vert
\Gamma \right\vert }f_{t}(u,i)\,dudi\text{ in probability }
\end{equation*}
\ uniformly for $t\in \lbrack 0,T]$ and $f_{t}=f_{\Lambda ,t} + f_{\partial
\Lambda ,t}$ satisfies the following differential equation: for $u\in \Gamma
,i\in S$%
\begin{eqnarray}
\frac{\partial }{\partial t}f_{\Lambda ,t}(u,i) &=&\tsum\limits_{k\in S}%
\mathbf{c}(u,k,i,f)f_{\Lambda ,t}(u,k)-f_{\Lambda ,t}(u,i)\tsum\limits_{k\in
S}\mathbf{c}(u,i,k,f)  \label{lf-de} \\
f_{0}(u,i) &=&f(u,i)  \notag
\end{eqnarray}
\end{theorem}

Note that $\mathbf{c}(u,k,i,f)=c(u,k,i,f_{\Lambda }+$\textbf{\ }$f_{\partial
\Lambda })$ is given by the similar formula to (\ref{conti-rate1}) with $%
\mathcal{J\ast }f(u)=\int_{\Gamma }\mathcal{J}(u-v)f(v)dv$ for $u\in \Lambda
;$ so the rates depend on $f_{\partial \Lambda }$ as well as $f_{\Lambda }.$

\subsection{Heuristic derivation of the differential equations}

In this section we justify, heuristically, the IDEs obtained in Theorems \ref%
{them-long-range-perBC} and \ref{thm-long-range-fixedBC}. For simplicity we
assume periodic boundary conditions but the other case is similar. The
differential equations (\ref{eq-de-per}) and (\ref{lf-de}) are examples of
input-output equations. In particular, by summing over the strategy set, it
is easy to see that $\tsum\nolimits_{i\in S}f_{t}(u,i)$ is independent of $t$
and therefore if $f_{0}\in \mathcal{M}$, then $f_{t}\in \mathcal{M}$ for all 
$t$. Also the space $\mathcal{M}$ can be thought of as a product over the
space of the standard strategy simplex $\Delta $ of game theory, i.e., $%
\mathcal{M}=\prod_{u\in \mathbf{T}^{d}}\ \Delta $. As shown in evolutionary
game theory textbooks \citep{Weibull95, Sandholm08} one can derive
heuristically the ODEs from corresponding stochastic processes. The main
assumption used there is that the rates depend only on the average
proportion of players with a given strategy. In this section we provide, for
the convenience of a reader, a similar heuristic derivation from microscopic
processes in the case of the spatial IDE (\ref{eq-de-per}); we replace
global average by spatially localized averages as expressed in the limit of
the empirical measure (\ref{empirical}).

For microscopic sites $x$ and $y,$ let us denote by $u=\gamma x$ and $%
v=\gamma y$ the corresponding spatial positions at the mesoscopic level. 
%let us assume that the agent at the (microscopic) site $x$ has strategy
%$i$ and the agent at the site $y$ has strategy $y$ and let us denote 
%$u:=\gamma x$ and $v:=\gamma y$ denote the corresponding positions 
%in the mesoscopic domain.  
For the sake of exposition let us suppose that $c^{\gamma }(x,\sigma ,k)$ is
given by%
\begin{equation*}
c^{\gamma }(x,\sigma ,k)=F(u(x,\sigma ,k)-u(x,\sigma ,\sigma (x))\,.
\end{equation*}%
For any continuous function $g$ on $\mathbf{T}^{d}\times S,$ by the
definition of the empirical measure (\ref{empirical}) we have the identity 
\begin{equation*}
\frac{1}{\left\vert \mathbf{T}^{d,\gamma }\right\vert }\sum_{x\in \mathbf{T}%
^{d,\gamma }}g(\gamma x,\sigma (x))=\int_{\mathbf{T}^{d}\times S}g(u,i)\pi
^{\gamma }(\sigma ,du,di)\,.
\end{equation*}%
Since $\left\vert \mathbf{T}^{d,\gamma }\right\vert \approx \gamma ^{-d}$
and if we \emph{assume} that $\pi ^{\gamma }\left( \sigma ,du,di\right)
\rightarrow f(u,i)dudi$, we obtain 
\begin{equation}
\lim_{\gamma \rightarrow 0}\sum_{x\in \mathbf{T}^{d,\gamma }}\gamma
^{d}g(\gamma x,\sigma \left( x\right) )=\int_{\mathbf{T}^{d}\times
S}g(u,i)f(u,i)dudi.  \label{eq-conv1}
\end{equation}%
Using (\ref{eq-conv1}), we find%
\begin{equation*}
\lim_{\gamma \rightarrow 0}\sum_{x\in \mathbf{T}^{d,\gamma }}\gamma ^{d}%
\mathcal{J}\ \left( \gamma (x-y)\right) a(k,\sigma (y))=\int_{\mathbf{T}%
^{d}\times S}a(k,l)\mathcal{J}(u-v)f(v,l)dvdl=\sum_{l\in S}a(k,l)\mathcal{%
J\ast }f(u,l).
\end{equation*}%
Therefore if $\sigma (x)=i$ we then obtain 
%So, formally, we obtain the function $\mathbf{c}$ in theorem \ref%%
%{them-long-range-perBC} from $c^{\gamma }$ by replacing $\gamma x,\gamma y,$ 
%$\sigma (x),$and $\sigma (y)$ by $u,v,i,$ and $l$:%
\begin{eqnarray*}
c^{\gamma }(x,\sigma ,k) &=&F(u(x,\sigma ,k)-u(x,\sigma ,\sigma (x)) \\
&=&F\left( \sum_{y\in \mathbf{T}^{d,\gamma }}\gamma ^{d}\mathcal{J}\ \left(
\gamma x-\gamma y\right) a(k,\sigma (y))-\sum_{y\in \mathbf{T}^{d,\gamma
}}\gamma ^{d}\mathcal{J}\ \left( \gamma x-\gamma y\right) a(\sigma
(x),\sigma (y))\right) \\
&\underset{\gamma \rightarrow 0}{\longrightarrow }&F\left( \sum_{l\in
S}a(k,l)\mathcal{J\ast }f(u,l)-\sum_{l\in S}a(i,l)\mathcal{J\ast }%
f(u,l)\right) =\mathbf{c}(u,i,k,f),
\end{eqnarray*}%
and this gives equation (\ref{conti-rate1}). After having identified rates,
we can now explain how to derive the IDE (\ref{eq-de-per}). We write 
\begin{equation*}
\left\langle \pi ^{\gamma },g\right\rangle (\sigma ):=\int_{\mathbf{T}%
^{d}\times S}g(u,i)\pi ^{\gamma }(\sigma ,dudi),\text{ \ \ \ }\left\langle
f,g\right\rangle :=\int_{\mathbf{T}^{d}\times S}g(u,i)f(u,i)dudi,
\end{equation*}%
where we view $\left\langle \pi ^{\gamma },g\right\rangle (\sigma )$ as a
function of the configuration $\sigma $. The action of the generator on this
function is 
\begin{equation}
L_{\gamma }\left\langle \pi ^{\gamma },g\right\rangle (\sigma )=\sum_{k\in
S}\int_{\mathbf{T}^{d}\times S}\mathbf{c}(u,i,k,\pi ^{\gamma }(\sigma
))\left( g(u,k)-g(u,i)\right) \pi ^{\gamma }(\sigma ,dudi)\,.  \notag
\end{equation}%
>From the martingale representation theorem for Markov processes 
\citep[for
example see][]{Ethier86} there exists a martingale $M_{t}^{g,\gamma }$ such
that 
\begin{equation}
\left\langle \pi _{t}^{\gamma },g\right\rangle =\left\langle \pi
_{0}^{\gamma },g\right\rangle +\int_{0}^{t}ds\sum_{k\in S}\int_{\mathbf{T}%
^{d}\times S}\mathbf{c}(u,i,k,\pi _{s}^{_{\gamma }})\left(
g(u,k)-g(u,i)\right) \pi _{s}^{_{\gamma }}(dudi)+M_{t}^{g,\gamma }.
\label{proof-rem2}
\end{equation}%
As $\gamma \rightarrow 0,$ one proves that $M_{t}^{g,\gamma }\rightarrow 0$.
Thus if $\pi _{t}^{_{\gamma }}(dudi)\rightarrow f(t,u,i)dudi$ as $\gamma
\rightarrow 0,$ equation $($\ref{proof-rem2}$)$ becomes 
\begin{equation*}
\ \left\langle f_{t},g\right\rangle =\left\langle f_{0},g\right\rangle
+\int_{0}^{t}ds\sum_{k\in S}\int_{\mathbf{T}^{d}\times S}\mathbf{c}%
(u,i,k,f_{s})\left( g(u,k)-g(u,i)\right) f_{s}(u,i)dudi
\end{equation*}%
and upon differentiating with respect to time, we find%
\begin{equation}
<\frac{\partial f_{t}}{\partial t},g>=\sum_{k\in S}\int_{\mathbf{T}%
^{d}\times S}\mathbf{c}(u,i,k,f_{t})\left( g(u,k)-g(u,i)\right)
f_{t}(u,i)dudi  \label{weak-de}
\end{equation}%
which is the weak formulation of the IDE (\ref{eq-de-per}) obtained by
integrating over $u$ and $i.$

The proof of Theorem \ref{them-long-range-perBC} and Theorem \ref%
{thm-long-range-fixedBC}, which we present in the appendix, is a variation
on the proof given in \citet{Comets87}, \citet{Kipnis99}, and %
\citet{Katsoulakis05}. Unlike these papers, in the case of non-innovative
dynamics studied here there is no detailed balance condition, however the
mesoscopic limit of the type (\ref{res-conv}) can still be carried out in
the Kac scaling (\ref{Kac-pot}). Using the martingale representation (\ref%
{proof-rem2}), we show that $\left\{ \mathbf{Q}^{\gamma }\right\} _{\gamma
}, $ a sequence of probability laws of $\left\{ \pi _{t}^{\gamma }\right\}
_{\gamma }$, is relatively compact. We then show that all the limit points
are concentrated on the weak solutions of (\ref{lf-de}) and on measures
absolutely continuous with respect to Lebesgue measure. Finally we
demonstrate that the weak solutions of (\ref{lf-de}) are unique so that we
conclude the convergence of $\mathbf{Q}^{\gamma }$ to the Dirac measure
concentrated on the solution of (\ref{lf-de}).

\subsection{Spatially uniform interactions: Mean-field Dynamics}

The goal of this section is to show that under the assumption of uniform
interactions the spatially aggregated process is still a Markov chain (such
process is called lumpable). Furthermore our IDEs reduce then to the usual
ODEs of evolutionary game theory, as it should be. The relationships between
the various processes and differential equations is illustrated in Figure %
\ref{codia}. 
\begin{figure}[tbp]
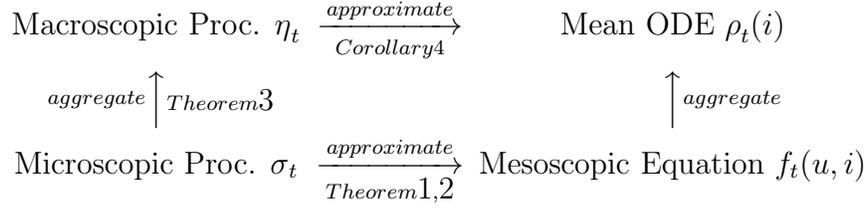

\begin{equation*}
\begin{CD} \textrm{Macroscopic Proc. } \eta_t @> approximate> Corollary 4 >
\textrm{Mean ODE } \rho_t(i) \\ @A aggregate A Theorem \ref{Th-Mean-Markov
copy(1)} A @ AA aggregate A \\ \textrm{Microscopic Proc. } \sigma_t @>
approximate > Theorem \ref{them-long-range-perBC},
\ref{thm-long-range-fixedBC} > \textrm {Mesoscopic Equation } f_t(u,i)
\end{CD}
\end{equation*}%
\caption{\textbf{The relationships between the microscopic process and the
macroscopic process and between the stochastic process and the deterministic
approximation.}}
\label{codia}
\end{figure}
Let us take periodic boundary conditions and uniform interactions, i.e., $%
\mathcal{J}\equiv 1$ on $\mathbf{T}^{d}$. Let us further define the
aggregate variables 
\begin{equation*}
\eta ^{\gamma }(i):=\frac{1}{\left\vert \mathbf{T}^{d,\gamma }\right\vert }%
\sum_{x\in \mathbf{T}^{d,\gamma }}\delta (\sigma (x),i)
\end{equation*}%
which counts the proportion of agents with strategy $i$ in the entire domain 
$\mathbf{T}^{d,\gamma }$. Note that this is obtained, equivalently, by
integrating the empirical measure $\pi ^{\gamma }(\sigma ,dudi)$ over the
spatial domain $\mathbf{T}^{d}.$ We observe that $\eta ^{\gamma }$ depends
on $\gamma $ only through the size of the domain$~n^{d}$ i.e., $n^{d}=\frac{1%
}{\gamma ^{d}}$ and $n^{d}\rightarrow \infty $ as $\gamma \rightarrow 0$.
Furthermore since $\mathcal{J}\equiv 1,$ the payoff $u(x,\sigma ,k)$ depends
on $\sigma $ only through the aggregated variable $\eta ^{n}(i)$. Indeed, we
have 
\begin{equation*}
u(x,\sigma ,k):=\frac{1}{n^{d}}\sum_{y\in \mathbf{T}^{d,n}}\sum_{l\in
S}\delta (\sigma (y),l)a(k,l)=\sum_{i\in S}a(k,i)\eta ^{n}(i)
\end{equation*}%
Thus for the strategy revision rates, if $\sigma (x)=j$ we define 
\begin{equation*}
c^{M}(j,k,\eta ^{n}):=c^{\gamma }(x,\sigma ,k)\,,
\end{equation*}%
since the right hand side is independent of $x$ and depends only on $\sigma $
through the corresponding aggregate variable $\eta ^{n}.$ Therefore $\eta
_{t}^{n}$ itself is a Markov process as we will show in Theorem \ref%
{Th-Mean-Markov copy(1)} below, and the state space for $\eta _{t}^{n}$ is
the discrete simplex 
\begin{equation*}
\Delta ^{n}=\left\{ \{\eta (i)\}_{i\in S}\,;\,\tsum_{i\in S}\eta
(i)=1\,\,,n^{d}\eta (i)\in \mathbb{N}_{+}\right\}
\end{equation*}%
To capture the transition induced by an agent's strategy switching, we write 
\begin{equation*}
\eta ^{j,k}(i)=\left\{ 
\begin{tabular}{ll}
$\eta (i)$ & if $i\neq k,j$ \\ 
$\eta (i)-\frac{1}{n^{d}}$ & if $i=j$ \\ 
$\eta (i)+\frac{1}{n^{d}}$ & if $i=k$%
\end{tabular}%
\right.
\end{equation*}%
Thus $\eta ^{j,k}$ is the state obtained from $\eta $ if one agent switches
his strategy from $j$ to $k$.

\begin{theorem}
\label{Th-Mean-Markov copy(1)} Suppose the interaction is uniform, then $%
\eta ^{n}$ is a Markov chain with state space $\Delta ^{n}$ and generator 
\begin{equation}
L^{M,n}g\left( \eta \right) =\sum_{k\in S}\sum_{j\in S}n^{d}\eta
^{n}(j)c(j,k,\eta )(g(\eta ^{n,j,k})-g(\eta ^{n}))\,.  \label{mean-gen}
\end{equation}
\end{theorem}

The factor $n^{d}$ in (\ref{mean-gen}) comes from the fact that in a time
interval of size 1, on average $n^{d}$ strategy switches take place, and
among those, $n^{d}\eta ^{n}(j)$ are switches from agents with type $j$. 
%For instance, when the rate is comparing and imitative, we find 
%\begin{equation*}
%c(j,k,\eta )=\eta (k)F(\sum_{l\in S}a(k,l)\eta (l)-\sum_{l\in S}a(j,l)\eta
%(l))
%\end{equation*}%
%The generators for other rates function are readily obtained following the
%similar approaches. 
Theorem \ref{Th-Mean-Markov copy(1)} shows that the stochastic process with
uniform interactions coincides with multi-type birth and death process in
population dynamics \citep{Blume98, Benaim03}. In addition, following %
\citet{Kurtz70}, \citet{Benaim03}, and \citet{Darling08}, or as a special
case of our result (Corollary \ref{thm-uniform} below) we can obtain mean
field ODEs. \ Furthermore, at the mesoscopic level, the IDEs reduce to the
usual ODEs of evolutionary game theory as follows (See Figure \ref{codia}).
We note that when $\mathcal{J\equiv }$ $1,$ we can define 
\begin{equation*}
\rho (i):=\int f(u,i)du=\mathcal{J}\ast f(i)
\end{equation*}%
so $\mathbf{c}(u,k,i,f)$ is independent of $u$ and this\ again allows to
define 
\begin{equation}
\mathbf{c}^{M}(k,i,\rho ):=\mathbf{c}(u,k,i,f)  \label{mean-c}
\end{equation}%
Thus, from the IDE (\ref{eq-de-per}) we obtain%
\begin{equation}
\frac{d\rho _{t}(i)}{dt}=\sum_{k\in S}\mathbf{c}^{M}(k,i,\rho )\rho
_{t}(k)-\rho _{t}(i)\sum_{k\in S}\mathbf{c}^{M}(i,k,\rho ).  \label{mean-ode}
\end{equation}%
For example, in the case of the comparing and imitative rate we have 
\begin{equation*}
\mathbf{c}^{M}(k,i,\rho )=\rho (i)F\left( \sum_{l\in S}a(i,l)\rho
(l)-\sum_{l\in S}a(k,l)\rho (l)\right) \,.
\end{equation*}%
If $F(s)=\frac{1}{\kappa }\log \left( \exp \left( \kappa s\right) +1\right)
, $ then $F(s)-F(-s)=s$ and (\ref{mean-ode}) becomes the (imitative)
replicator dynamics. Other well-known mean field ODEs, such as logit
dynamics and Smith dynamics, are similarly derived by choosing appropriate $%
F $. Finally, as a consequence of Theorem \ref{them-long-range-perBC} we
have the following corollary which is the \emph{continuous-time} version of %
\citet{Benaim03}'s result. To state the result, we write $\left\Vert \eta
^{n}\right\Vert _{u}:=\sup\nolimits_{i\in S}\left\vert \eta
^{n}(i)\right\vert .$

\begin{corollary}[Uniform Interaction; Benaim and Weibull, 2003]
\label{thm-uniform} Suppose that the interaction is uniform and that the
strategy revision rate satisfies $\mathbf{C1-C3}$. Suppose there exists $%
\rho \in \Delta $ such that the initial condition $\eta _{0}^{n}$ satisfies 
\begin{equation*}
\lim_{n\rightarrow \infty }\eta _{0\text{ }}^{n}\,=\,\rho \text{ in
probability }
\end{equation*}%
%
%
%
%
%
%
%
%
%
%
%
%
%
%
%
%
%
%
%\left\Vert \eta _{0\text{ }}^{\gamma }-\rho \right\Vert _{u}\rightarrow 0%
%\text{ for }a.e\text{. }\omega \text{ \ as }\gamma \rightarrow 0
%\end{equation*}%
%for a given $\rho $ $\in $ $\Delta (S).$ 
Then for every $T>0$%
\begin{equation}
\lim_{n\rightarrow \infty }\eta _{t}^{n}(i)\longrightarrow \rho _{t}(i)\text{
in probability }
\end{equation}%
\ uniformly for $t\in \lbrack 0,T]$ and $\rho _{t}(i)$ satisfies the
following differential equation: for $i\in S$%
\begin{eqnarray}
\frac{d\rho _{t}(i)}{dt} &=&\sum_{k\in S}\mathbf{c}^{M}(k,i,\rho )\rho
_{t}(k)-\rho _{t}(i)\sum_{k\in S}\mathbf{c}^{M}(i,k,\rho )  \label{mfode} \\
\rho _{0}(i) &=&\rho (i)
\end{eqnarray}%
where $\mathbf{c}^{M}$ is given by (\ref{mean-c}). Moreover, there exist $C$
and $\epsilon _{0\text{ }}$such that $\ $for all $\epsilon \leq \epsilon
_{0},$ there exists $n_{0}$ such that for all $n\geq n_{0}$%
\begin{equation}
P\left\{ \sup_{t\leq T}\left\Vert \eta _{t}^{n}-\rho _{t}\right\Vert
_{u}\geq \epsilon \right\} \leq 2\left\vert S\right\vert e^{-\frac{%
n^{d}\epsilon ^{2}}{TC}}\,.  \label{exponential-estimate}
\end{equation}
\end{corollary}

Estimates such as (\ref{exponential-estimate}) describe the validity regimes
of the approximation by mean field models (\ref{mfode}) in terms both of
agent number $n$ and the time window $[0, T]$

\section{Equilibrium Selection and Pattern Formation}

% in Deterministic Equations}

In this section we illustrate the usefulness and the versatility of the
IDE's derived in Section \ref{main-theorem} by using a combination of linear
analysis and numerical simulations. We will consider the following equations
(see the rates in Section \ref{model-section} and at the beginning of the
Appendix)

\medskip

\noindent \textbf{(a) Logit/Glauber dynamics:} If the rate is given by (\ref%
{logit}) we obtain the IDE 
\begin{equation*}
\frac{\partial }{\partial t}f_{t}(u,i)\,=\frac{\exp \left( \sum_{l\in
S}a(i,l)\mathcal{J}\ast f_{t}(u,l)\right) }{\sum_{k\in S}\exp \left(
\sum_{l\in S}a(k,l)\mathcal{J}\ast f_{t}(u,l)\right) }-f_{t}(u,i)\,
\end{equation*}%
which generalizes the well-known logit ODE of game theory.

\medskip

\noindent \textbf{(b) Imitative replicator equation:} Let us suppose that
the rates are given by equation (\ref{rep-imitative}). Then we obtain 
\begin{eqnarray}
\frac{\partial }{\partial t}f_{t}(u,i) \,&=&\, \sum_{k \in S} \left[ f(u,k) 
\mathcal{J}\ast f(u, i) F\left( \sum_{l\in S} (a(i,l)-a(k,l)) \mathcal{J}
\ast f_t(u,l)\right) \right.  \notag \\
&& \hspace{1cm} \left. - f(u,i) \mathcal{J} \ast f(u, k) F\left( \sum_{l\in
S} (a(k,l)-a(i,l)) \mathcal{J} \ast f_t(u,l)\right) \right]  \label{imrep}
\end{eqnarray}
Note that the equation depends explicitly on $F$. This is to be contrasted
with the replicator ODE which is independent of $F$ whenever $F$ satisfies
the relation $F(t)-F(-t)=t$. This is a purely spatial effect: indeed if we
take $f(u,i)$ independent of $u$ for all $i$ then equation (\ref{imrep})
reduces to the replicator ODE.

\noindent \textbf{(c) Biological replicator equation:} Note that one can
also derive a \textquotedblleft replicator IDE\textquotedblright\ using a
\textquotedblleft biological fitness\textquotedblright argument, i.e., the
rate of change in the population of a given type is proportional to the
difference between the fitness of this type and the average fitness in the
population: %  According  to  our model to evaluate payoff we find 
\begin{equation*}
\frac{\partial }{\partial t}f_{t}(u,i)\,=\,f_{t}(u,i)\left[ \sum_{l\in
S}a(i,l)\mathcal{J}\ast f(u,l)-\sum_{k,l\in S}f(u,k)a(k,l)\mathcal{J}\ast
f(u,l)\right]
\end{equation*}%
This equation, while it still lacks a convincing derivation from a
microscopic stochastic process is an interesting equation in itself and it
shares many of the nice properties of the replicator ODEs.

\subsection{Spatio-temporal Linear Stability}

In this section we present the linear stability analysis of IDEs around
stationary solutions 
%which provides a powerful technique for analyzing nonlinear IDEs around stationary
as a first step to understand the generation and propagation of temporal and
spatial morphologies; we refer to \cite{murray89} for numerous examples and
applications of linear stability analysis of partial differential equation
models. Let us consider the following general type of integro-differential
equations: 
\begin{equation}
\left\{ 
\begin{tabular}{ll}
$\frac{\partial f}{\partial t}=F(\mathcal{J\ast }f,f)$ & in \ $\Lambda
\times (0,T]$ \\ 
$f(0,x)=f^{0}(x)$ & on $\Lambda \times \{0\}$%
\end{tabular}%
\right. ,  \label{IDE1}
\end{equation}%
where $\Lambda \subseteq \mathbb{R}^{d}$ or $\Lambda =\mathbf{T}^{d}$, $f\in 
\mathcal{M}(\Lambda \times S),\mathcal{J\ast }f:=(\mathcal{J}\ast f_{1},$ $%
\mathcal{J}\ast f_{2},\cdots ,$ $\mathcal{J}\ast f_{n})^{T},$ and $F$ is
smooth in both arguments. First, observe that if $f$ is spatially
homogeneous, i.e., $f(u,t)=f(t)$, %does not depend on $u$  
then $\mathcal{J}\ast f=f(\mathcal{J}\ast 1)=f$, and thus the IDE (\ref{IDE1}%
) reduces to the ODE 
\begin{equation*}
\frac{\partial f}{\partial t}=F(f,f)\,.
\end{equation*}%
This ODE, in turn, is exactly the ODE obtained if the interactions are
uniform $\mathcal{J}\equiv \text{const}$. This shows that the spatially
homogenous solutions of (\ref{IDE1}) are exactly the stationary solutions of
the corresponding mean-field ODE. In particular every spatially homogenous
stationary solution $f_{0}$, satisfies $F(f_{0},f_{0})=0$. We record this
observation in Lemma \ref{lem:siss}.

\begin{lemma}[Space Independent Stationary Solutions]
\label{lem:siss}$f_{0}$ is a spatially independent stationary solution to (%
\ref{IDE1}) if and only if $F(f_{0},f_{0})=0$.
\end{lemma}

Next we study spatiotemporal perturbations of such constant states by
linearizing around a spatially homogeneous stationary solution, $f_{0}$: let 
$f=f_{0}+\epsilon D$ where $D=D(u,t)$ and substituting into (\ref{IDE1}), we
obtain 
\begin{equation}
\epsilon \frac{\partial D}{\partial t}=F(f_{0}+\epsilon \mathcal{J}\ast
D,f_{0}+\epsilon D).  \label{eq_var_bef}
\end{equation}%
For small $\epsilon $ we expand the right hand side of equation (\ref%
{eq_var_bef}) around $\epsilon =0$, ignore the terms of order $\epsilon ^{2}$
or smaller and obtain 
\begin{equation}
\frac{\partial D}{\partial t}=M\mathcal{J}\ast D+ND  \label{var-eq}
\end{equation}%
where $\left( M\right) _{i,j}:$= $\frac{\partial F_{i}}{\partial r_{j}}$ , $%
\left( N\right) _{i,j}$:= $\frac{\partial F_{i}}{\partial s_{j}},$ and each
derivative is evaluated at $(f_{0},f_{0}).$ We can solve (\ref{var-eq})
explicitly using Fourier transform (see the appendix for details) and obtain
the following

\bigskip

\noindent \textbf{Dispersion Relation:} 
%\begin{lemma}[Linear Stability Analysis]
Eigenvalues for the solutions to (\ref{var-eq}) are given by 
\begin{equation}
\lambda (k)=\text{eigenvalue}(M\mathcal{\hat{J}(}k)+N)\text{ for }k\in 
\mathbb{Z}^{d}  \label{gen-dispersion}
\end{equation}%
where $\mathcal{\hat{J}}\left( k\right) =\int_{\mathbf{T}^{d}}\mathcal{J}
(u)e^{2\pi ik\cdot u}du$ are the Fourier coefficients of $\mathcal{J}$. 
%\end{lemma}

\vskip .4cm

In general, dispersion relations are a useful tool to investigate the
generation and early-stage propagation of spatial phenomena for nonlinear
PDE or IDE, see for instance \cite{murray89}. In our case (\ref%
{gen-dispersion}) provides the growth (or decay) rates of approximate
solutions to equation (\ref{eq_var_bef}). As we will see later, identifying
the Fourier coefficients $k$ which are linearly unstable (i.e., $\lambda
(k)>0$) allows us to identify the regions in phase space where instabilities
occur and could lead, coupled with the nonlinear effects, to the formation
of complex spatial structures. Finally, the linear stability analysis
provides us computational benchmarks for our simulations.

\subsection{Example: Two-strategy symmetric coordination games}

We consider two-strategy symmetric coordination games with payoffs (\ref%
{avpayoff}) being normalized in a way that $a(1,2)=a(2,1)=0$ and $a(1,1)>0$, 
$a(2,2)>0$. If $p(u)\equiv f(u,1)$, using that $f(u,1)+f(u,2)=1$ we can
write a single equation for $p(u)$ and obtain an equation of the form (\ref%
{IDE1}) with 
\begin{eqnarray}
\text{\textbf{Replicator IDE} \ }F_{R}(r,s) &:&=(1-s)rF_{\kappa }\left(
\beta \left( r-\zeta \right) \right) -s(1-r)F_{\kappa }\left( \beta \left(
\zeta -r\right) \right)  \label{rep-IDE-F} \\
\text{\textbf{Logit IDE} \ \ \ \ \ \ \ \ \ }F_{L}(r,s) &:&=l(\beta \left(
r-\zeta \right) )-s  \label{logit-IDE-F}
\end{eqnarray}%
where $\zeta =\frac{a(2,2)}{a(1,1)+a_(2,2)}$\ , $\beta =a(1,1)+a(2,2),$ $%
l\left( t\right) :=\frac{1}{1+\exp \left( -t\right) },$\textsl{\ }and $%
F_{\kappa }(t):=\frac{1}{\kappa }\log \left( \exp (\kappa t)+1\right) $
(recall equation (\ref{reg_rep}))$.$ In equations (\ref{rep-IDE-F}) and (\ref%
{logit-IDE-F}) $r$ and $s$ are variables representing $\mathcal{J\ast }p$
and $p$, respectively. Note that $\zeta $ is the mixed strategy Nash
equilibrium and $\beta $ is positive.

We refer to (\ref{rep-IDE-F}) at $\kappa =\infty $ as a replicator IDE,
while we also consider the regularized replicator IDE (\ref{rep-IDE-F}) for $%
\kappa <\infty $, and refer to (\ref{logit-IDE-F}) as a logit IDE. In
addition to the conditions for $\mathcal{J\ }$stated in Section 2.3, we
assume that $\mathcal{J\ }$is symmetric: $\mathcal{J}(x)=\mathcal{J}(-x)$
for $x\in \Lambda .$

\subsubsection{Stationary solutions and their linear stability}

To find spatially homogenous stationary solutions, we need to set $%
F_{R}\left( p,p\right) =0$ and $F_{L}\left( p,p\right) =0$. $\ $Then, for
the replicator case $p=0,1,$ and $\zeta $ are three stationary solutions. In
the case of logit dynamics, using $l(\kappa z)=\frac{1}{2}+\frac{1}{2}\tanh
(\kappa \frac{z}{2})$ and changing the variable, $p\mapsto 2p-1:=u,$ the
differential equation becomes%
\begin{equation}
\frac{\partial u}{\partial t}=-u+\tanh (\frac{\beta }{4}(\mathcal{J}\ast
u+(1-2\zeta )))  \label{eq_ising}
\end{equation}%
which is the well-known Glauber mesoscopic equation 
\citep{DeMasi94, Katsoulakis97,
Presutti09} with $\beta $ being the \emph{inverse temperature}. All known
results for (\ref{eq_ising}), such as the existence of traveling wave
solutions in one space dimension and the geometric evolution of interfaces
between homogeneous stationary states in higher dimensions, are directly
applicable to the logit dynamics. Because of this connection, we have the
following characterization of stationary solutions to logit dynamics; the
proof is the consequence of (\ref{eq_ising}) and the analysis of Glauber
dynamics \citep{Presutti09} or it can easily be done directly.

\begin{lemma}
\label{logit-space} Suppose that the game is a coordination game. Then,
there exists $\beta _{C\text{ }}$ such that for $\beta <\beta _{C}$ there
exists one spatially homogenous stationary solution, $p_{1},$ and for $\beta
>\beta _{C}$ \ there exist three spatially homogenous stationary solutions, $%
p_{1},$ $p_{2},$ and $p_{3}$.
\end{lemma}

We note here the different role of $\beta $ in each one of the IDEs (\ref%
{rep-IDE-F}) and (\ref{logit-IDE-F}). \ Since $\beta =a_{11}+a_{22},$ $\beta 
$ measures the size of payoffs in coordination games, capturing the
importance of the game to the players; as $\beta \rightarrow 0,$ the payoffs
become negligible. In replicator IDEs, a change in $\beta $ merely
corresponds to a time change in IDEs, as we can easily see from equations (%
\ref{rep-IDE-F}). As the game become less important, the replicator system
evolves slowly; when the game is for high stakes, the individual's updating
of strategy and, hence, the time evolution of the system is very fast. By
contrast, in logit dynamics $\beta $ becomes a parameter capturing the noise
level. So when $\beta $ is small (high noise), disorder pervades and the
system converges everywhere to $\frac{1}{2};$ everyone randomizes between
two strategies regardless of the payoffs. As $\beta $ gets\ higher (less
noise), logit dynamics approach best response dynamics, and the payoffs
weigh more. In this situation, the solution converges everywhere to a Nash
equilibrium.

Next we examine the linear stability of these stationary solutions. By
differentiating $F_{R},$ $F_{L}$, we find similarly to (\ref{gen-dispersion}%
) the dispersion relations for the replicator IDE: 
\begin{table}[h]
\centering%
\begin{tabular}{ll}
$p=0$ & $\lambda _{R}(k)=F_{\kappa }\left( -\beta \zeta \right) \mathcal{%
\hat{J}}(k)-F_{\kappa }\left( \beta \zeta \right) $ \\ 
$p=1$ & $\lambda _{R}(k)=F_{\kappa }\left( \beta \left( \zeta -1\right)
\right) \mathcal{\hat{J}}(k)-F_{\kappa }\left( \beta \left( 1-\zeta \right)
\right) $ \\ 
$p=\zeta $ & $\lambda _{R}\left( k\right) =\left( \frac{\log (2)}{\kappa }%
+\beta \zeta \left( 1-\zeta \right) \right) \mathcal{\hat{J}}(k)-\frac{\log
\left( 2\right) }{\kappa }$%
\end{tabular}%
\caption{\textbf{Dispersion Relations for the Replicator IDE}}
\label{tab-dispersion-rep}
\end{table}
\newline
Note that by our assumptions of $\mathcal{J}$, $\mathcal{\hat{J}(}k)$ is
real-valued and \TEXTsymbol{\vert}$\mathcal{\hat{J}}(k)|<1$ for all $k.$
Using this fact, we obtain Proposition \ref{prop:linstab_rep}.

\begin{proposition}[Linear Stability for the Replicator IDE]
\label{prop:linstab_rep}$p=0,1$ are linearly stable for the replicator
dynamics for coordination games.
\end{proposition}

Figure \ref{fig-dispersion} shows one example of the dispersion relations
for $p=\zeta .$ 
\begin{figure}[tb]
\centering 
\includegraphics[scale=0.6]{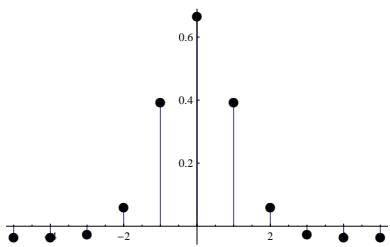}
\caption{\textbf{Dispersion Relation for the mixed strategy equilibrium in
regularized replicator dynamics.} {\protect\footnotesize {The figure shows
the dispersion relation $\protect\lambda _{R}(k)$\ at $p=\protect\zeta $.\ $%
\mathcal{J}(x)=\exp \left( -bx^{2}\right) /\protect\int \exp \left(
-bx^{2}\right) dx$,\ $b=20,\ \protect\kappa =20,~\protect\beta =3,\protect%
\zeta =\frac{1}{3}.$}}}
\label{fig-dispersion}
\end{figure}
Observe that $\lambda \left( k\right) >0$ for $k=0,$ $\pm 1,$ $\pm 2$ and
the solutions to linear equation (\ref{var-eq}) is of the form, $e^{2\pi
ik\cdot x}$ (see appendix)$.$ So, when $k=0,$ the corresponding solution is
constant along space and the eigenvalue $\lambda (0)$ is the eigenvalue for
the linearized equation of the mean-field ODE (\ref{mfode}). Thus $\lambda
(0)>0$ merely shows that $\zeta $ is unstable in mean-field ODE, and when $%
k=0$ we do not expect to observe any non-trivial spatial morphologies. At $%
k= $ $\pm 1,$ the corresponding solution has a period 1, involving $\cos (x)$%
, $\sin (x)$ or both and this solution may grow fast, dominating other
solutions with different frequencies. Note that the nonlinearity of the
replicator IDE implies a bound on the solutions, so that they remain in the
simplex, at each spatial location. An initially fast growing solution may be
bounded due to the nonlinearity effects and, hence, may develop to a
spatially heterogeneous solution. This is how we obtain the pattern
formation in Figure \ref{fig-pattern} (upper panels). For $k=$ $\pm 2,$ we
expect a similar spatial phenomenon, but now the solution involves $\cos
(2x) $ or $\sin (2x).$ Hence, we anticipate a finer pattern and, indeed,
observe this in the numerical simulation of Figure \ref{fig-pattern} (lower
panels).

In the case of logit dynamics, we note that $l^{\prime }(t)=l(t)\left(
1-l(t)\right)$, hence we easily obtain the dispersion relation for any
stationary solution, $p_{0}$:%
\begin{equation}  \label{linearlogit}
\lambda _{L}(k)=\beta (1-p_{0})p_{0}\mathcal{\hat{J}(}k)-1,\,\text{\ }k\in 
\mathbb{Z}^{d}
\end{equation}

\begin{proposition}[Linear Stability for the logit IDEs]
\label{prop-linear-log}Suppose that $0<\mathcal{\hat{J}(}k)$ for all $k.$
When $\beta <\beta _{C},$ the unique stationary solution $p_{0}$ is linearly
stable. When $\beta >\beta _{C}$, two stationary solutions, $p_{1},$ $p_{3},$
are linearly stable where three stationary solutions $p_{1},p_{2}$, and $p_{3%
\text{ }}$ are arranged in $p_{3}<p_{2}<p_{1}.$
\end{proposition}

We note that the Gaussian kernel satisfies the hypothesis, $0<\mathcal{\hat{J%
}(}k)$ for all $k.$ From table \ref{tab-dispersion-rep}, we see that the
dispersion relation for $p=\zeta $ in the replicator IDEs approaches 
\begin{equation*}
\lambda _{R}(k)=\beta \zeta (1-\zeta )\mathcal{\hat{J}(}k)\text{ as }\kappa
\rightarrow \infty
\end{equation*}%
and, as a result, $\lambda _{L}(k)$ is less than $\lambda _{R}(k)$ at $%
\kappa =\infty .$ Thus, we expect that the unstable steady solutions of the
replicator IDE are `more unstable' than the corresponding ones for logit
dynamics, and developed patterns in the replicator case may persist longer;
this conjecture is numerically confirmed in Figure \ref{fig:figure11} below.

\medskip \noindent \textbf{Remark:} Overall, the linearized analysis
depicted in the dispersion relations in Table 1 or in (\ref{linearlogit})
for the deterministic mesoscopic IDE, allows us to easily create a \emph{%
phase diagram} for pattern generation and strategies segregation, i.e. a
systematic representation of the parameter regimes of the microscopic models
for which we expect to have nontrivial spatial structures. On the other hand
such calculations, typically referred to in the engineering literature as
`systems tasks', are prohibitive using conventional Monte Carlo simulations
such as the ones in \cite{Szabo07}, for the complex microscopic processes in
Section \ref{model-section}; this is due not only to the expense of spatial
Monte Carlo simulations with many agents and strategies, but even more
importantly to the large number of parameters involved in the microscopic
models. %At
%the microscopic level, we only observe the stochastic fluctuations of
%strategy profiles rather than the smooth interfaces between patterns.

\subsubsection{Equilibrium selection: Mean-Field ODE versus Mesoscopic IDE}

To understand the importance of spatial interactions in our model compared
to existing mean-field models, we investigate the problem of equilibrium
selection for mean-field ODEs and spatial, mesoscopic IDEs. We ran numerical
simulations of the solutions to the spatial IDEs and the corresponding
mean-field ODEs whose initial values are chosen to be the spatial averages
of the initial data to the IDEs.

\begin{figure}[tb]
\centering 
\includegraphics[scale=0.5]{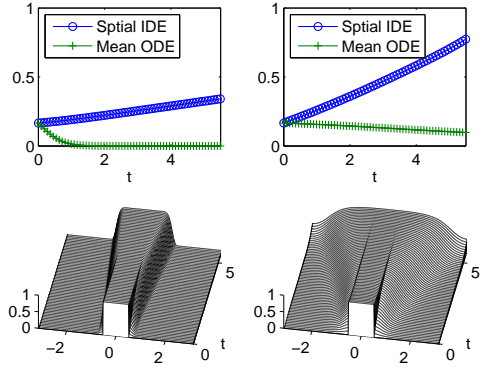}
\caption{\textbf{Comparison of equilibrium selections in mean-field ODEs and
IDEs (Periodic BC).} {\protect\footnotesize {The upper left panel shows
population densities of strategy 1 for the mean replicator ODE and IDE. The
upper right panel depicts the case of logit rule. The bottom panels show the
actual solutions of IDEs which were used for the comparison. $N=512.$\ $%
\Lambda =[-\protect\pi ,\protect\pi ].$\ $dt=0.001/(0.25N^{2}),$\ $a_{11}=%
\frac{20}{3},$\ $a_{22}=\frac{10}{3},~a_{12}=a_{21}=0$. $b=2~$for the
Gaussian kernel. The initial density in the upper panel is $\frac{1}{6}$ and
the initial datum for IDEs is $\mathbf{1}\{-\frac{\protect\pi }{6}<x<\frac{%
\protect\pi }{6}\}$.}}}
\label{fig:figure_add}
\end{figure}

In the upper panels in Figure \ref{fig:figure_add}, we present the
comparisons of population densities of strategy 1 in the coordination game
with payoffs, $a(1,1)=\frac{20}{3},$ $a(2,2)=\frac{10}{3}$, $a(1,2)=a(2,2)=0$
. The values for IDE are computed by integrating the corresponding spatial
solutions over space. The bottom panels show the evolutions of the spatial
solutions which were used to generate the upper panels. We used both the
replicator and logit equations; the left panels correspond to the replicator
equation, and the right panels describe the logit dynamics. As we see from
the upper panels, the solutions to mean-field ODEs converge to the
equilibrium where everyone in the population coordinates to strategy 2,
since the initial density $\frac{1}{6}$ belongs to the basin of the
attraction of this equilibrium. However, a small island of 1-strategists in
the spatial domain induces a transition toward an equilibrium of
coordinating to strategy 1, even though the total population density using
strategy 1 is still $\frac{1}{6}.$ In the replicator IDE case, the system
reaches a \emph{metastable state} $-$ a state where both strategies coexist
for a very long time $-$ and form a 1-dimensional pattern similar to Figure %
\ref{fig-pattern}. In case of the logit IDE, the propagation of strategy 1
is much faster than the replicator IDE, and typically the system converges
to the equilibrium of coordination to strategy 1. Heuristically, this is
because agents located near the island of 1-strategists, but are playing
strategy 2, face a roughly $50\%$ chance of interacting with 1-stategists
and $50\%$ chance of interacting with 2-strategists and, since strategy 1
yields a higher payoff than strategy 2, these agents are better off by
adopting strategy 1. This mechanism propagates strategy 1 to the whole
spatial domain in IDEs.

In similar simulations, though not reported in the paper, we have observed
that mean-field ODEs tend to overestimate the speed of convergence to
equilibrium. The mean-field ODE systems converge to equilibrium
exponentially fast, while in the IDEs the convergence is much slower, which
represents the fact that patterns are metastable; for related metastable
behavior for scalar reaction-diffusion equations we refer for instance to 
%TCIMACRO{\TeXButton{\citet{Carr89}}{\citet{Carr89}}}%
%BeginExpansion
\citet{Carr89}%
%EndExpansion
.Thus, one needs to exercise caution in studying equilibrium selection and
the convergence of the system using mean-field equations, especially when
the spatial consideration of system is important.

\subsubsection{Traveling front solution as a way of equilibrium selection:
Imitation versus Perturbed Best Responses}

Suppose that the domain is a subset of $\mathbb{R}$ with the fixed boundary
conditions or the whole real line $\mathbb{R}$. Then, this provides a
natural setting to study traveling front solutions, see for instance Figure %
\ref{fig:travel-front}. A solution is called a traveling front or wave
solution if it moves at a constant speed: i.e., a traveling front solution $%
p(x,t)$ can be written as $P(x-ct)$ for some constant $c$ and some function $%
P.$ The existence of traveling front solutions for the logit dynamics is the
direct consequences of known results for the Glauber equations. When $\zeta =%
\frac{1}{2},$ the existence of a unique standing wave (i.e. $c=0)$ was
proved and when there are three equilibrium states, the existence of
traveling waves was established \citep{DalPasso91, DeMasi95,
Orlandi97}. Particularly, if $\zeta <\frac{1}{2}$ one can find a solution
that satisfies $P\left( -\infty \right) =0$ and $P\left( \infty \right) =1$,
and travels at a negative speed. Thus the value of $P\left( \infty \right) $
propagates to the whole real line and as $t\rightarrow \infty ,$ the
solution becomes 1 everywhere; coordination to a state with the higher
payoffs becomes a dominating behavior. However, there is no existing
rigorous result, so far, on the replicator IDE, though we have observed this
solution in numerical simulations.

To compare the traveling wave solutions for each mesoscopic dynamics, we
first study the shapes of the standing waves. This is because the shapes of
the standing waves may depend on how \textquotedblleft
diffusive\textquotedblright\ the system is and the diffusiveness of the
system may, in turn, determine the speed of traveling waves. As in the usual
analysis of Allen-Cahn type PDE and Glauber IDE, we believe that the
sharpness of the standing wave varies with the diffusion effect of the
equations and the more \textquotedblleft diffusive\textquotedblright\ the
system is, the faster interfaces move \citep{Carr89, Katsoulakis97}.

\begin{figure}[tb]
\centering 
\includegraphics[scale=0.5]{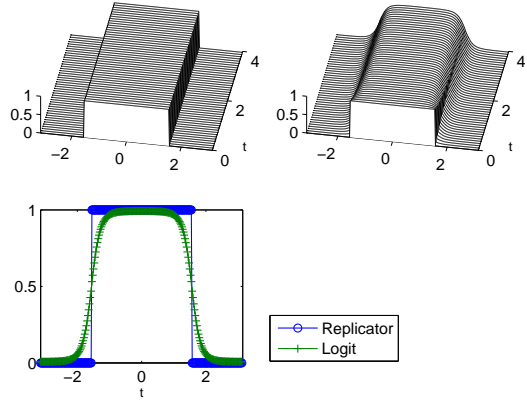}
\caption{\textbf{Comparison of standing waves between the replicator and the
logit dynamic (}$a_{11}=a_{22},$ \textbf{Periodic BC).} 
{\protect\footnotesize {\ The upper left panel shows the time evolution of
the population density of strategy 1 in the replicator dynamic. The upper
right panel describes the case of the logit dynamic. The bottom panel shows
the shapes of standing waves in both cases at time $4$.\ We consider the
replicator with $\protect\kappa =\infty .$\ {} $N=256.$\ $\Lambda =[-\protect%
\pi ,\protect\pi ].$\ $dt=0.001/(0.25N^{2}),$\ $a_{11}=5,$\ $%
a_{22}=5,~a_{12}=a_{21}=0$. \ $b=2.$\ The initial datum is }}$\mathbf{1}$%
{\protect\footnotesize {$_{[-\frac{1}{2}\protect\pi ,\frac{1}{2}\protect\pi %
]}$}}}
\label{fig:figure9}
\end{figure}

As Figure \ref{fig:figure9} shows, the shape of the standing wave in the
replicator dynamics with $\kappa =\infty $ is much sharper than that of the
logit dynamics. In other numerical simulations, we have observed that the
shape of the regularized replicator dynamics depend on $\kappa $; as $\kappa 
$ become larger, the shape is getting sharper. Since $F_{\kappa
}(t)\rightarrow \left[ t\right] _{+}$ as $\kappa \rightarrow \infty $, as $%
\kappa $ increases marginal gains from switching to a different strategy
become higher in response to increases in the payoff of that strategy; in
particular, at $\kappa =\infty $, this marginal gain becomes infinity. Thus
in the replicator IDEs of high payoffs, there is a zero probability for
actions against the optimal choice, hence the interface is very sharp.
However, the players in the logit dynamics do not have zero probabilities
for doing such an action when an agent is right on the \textquotedblleft
interface\textquotedblright ; i.e., there is a nonzero probability to select
something not optimal. That creates the \textquotedblleft
mushy\textquotedblright\ mixed region of a transition, see the schematic in
Figure \ref{fig:traveling} (b). From this observation we infer that the
logit dynamic is more \textquotedblleft diffusive\textquotedblright\ than
the replicator dynamic with $\kappa =\infty $; hence the interfaces in the
logit IDEs would move faster than those in the replicator IDEs. This is
numerically exhibited in Figure \ref{fig:figure10}.

\begin{figure}[tb]
\centering 
\includegraphics[scale=0.4]{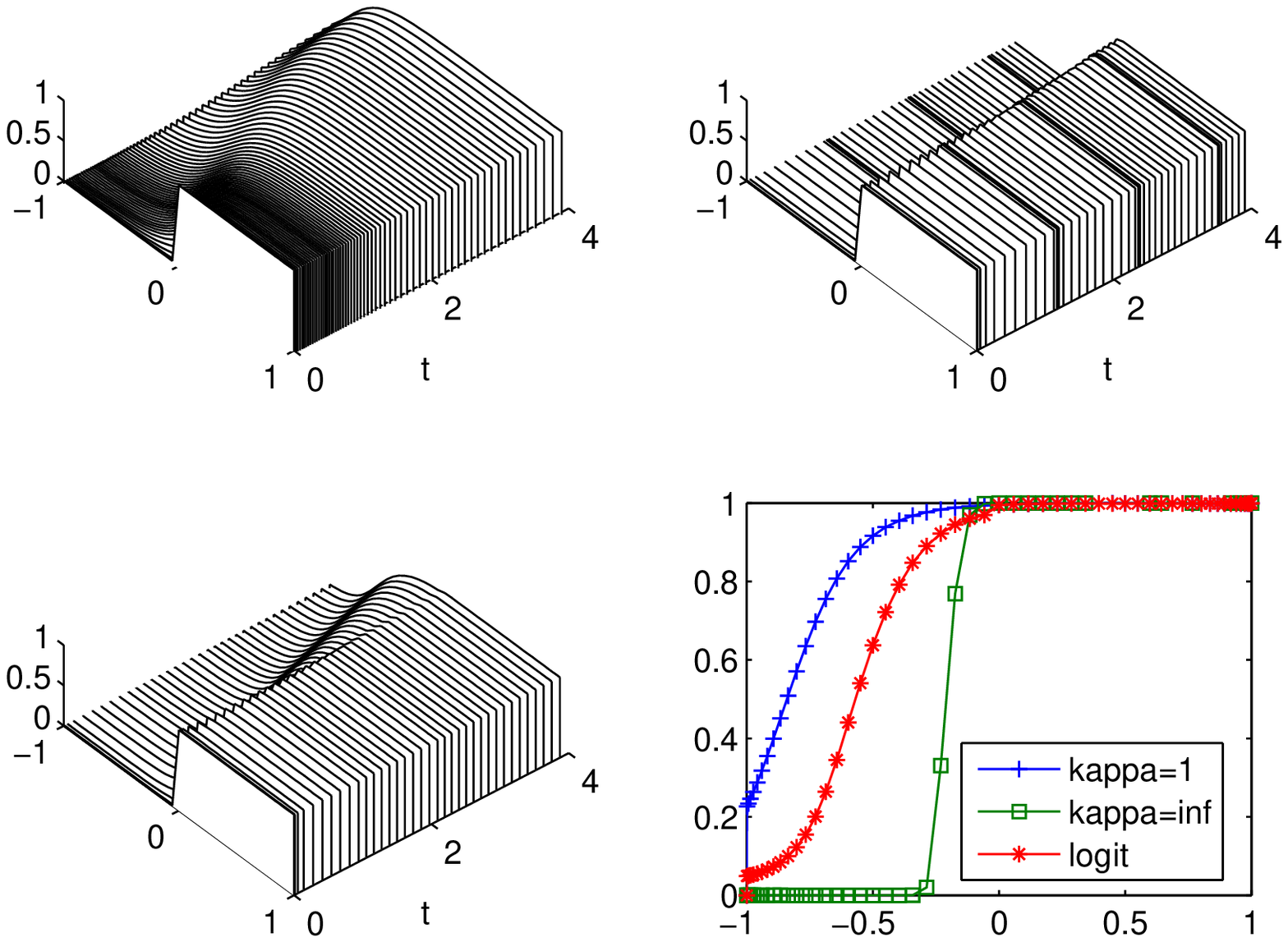}
\caption{\textbf{Comparison of traveling waves between the replicator and
the logit IDEs (Fixed BC).} {\protect\footnotesize {The upper panels show
the time paths of the population densities for strategy 1 in the replicator
with $\protect\kappa =1 $ (left) and the one with $\protect\kappa =\infty $
(right). The lower left panel shows the case of the logit dynamic. In the
bottom right panel we show the shapes of traveling waves at time 4. $N=256.$ 
$\Lambda =[-1,1]$, $\partial \Lambda =[-3,-1]\cup \lbrack 1,3]$ with the
fixed boundary condition $p(x)=0$ for $x\in \lbrack -3,-1]$ and $p(x)=1$ for 
$x\in \lbrack 1,3],~dt=0.001/(0.05N^{2}),$ $a_{11}=20/3,a_{12}=a_{21}=0. $ $%
b=2$ for the Gaussian kernel. The initial datum is $\mathbf{1}_{[0,1]}.$ }}}
\label{fig:figure10}
\end{figure}

We note that in the coordination game used for Figure \ref{fig:figure10},
the equilibrium of coordination to strategy 1 is the one predicted by the
existing equilibrium selection theories 
\citep{Harsanyi88,  Young98,
Hofbauer97, HofbauerWave97}. Particularly \citet{HofbauerWave97} shows,
under the best response dynamics, the existence of a traveling wave solution
which drives out the equilibrium of strategy 2, and at the same time
propagates the equilibrium of strategy 1. Although we observe the existence
of similar traveling wave solutions under various dynamics, the speed of
traveling varies dramatically. As Figure 9 shows the transition is extremely
slow in the replicator equation with $\kappa =\infty .$ So, when the society
is characterized by imitative behaviors and marginal gains from switching is
high, our model predicts that the transition to a \textquotedblleft better
equilibrium\textquotedblright\ is very slow and it takes a long time for
equilibrium selection to occur.

Finally we present another comparison between the imitative behavior with
the perturbed best response rule using unequal payoff coordination games $%
(a_{11}>a_{22})$ with the periodic boundary condition (Figure \ref%
{fig:figure11}). 
\begin{figure}[tb]
\centering 
\includegraphics[scale=0.35]{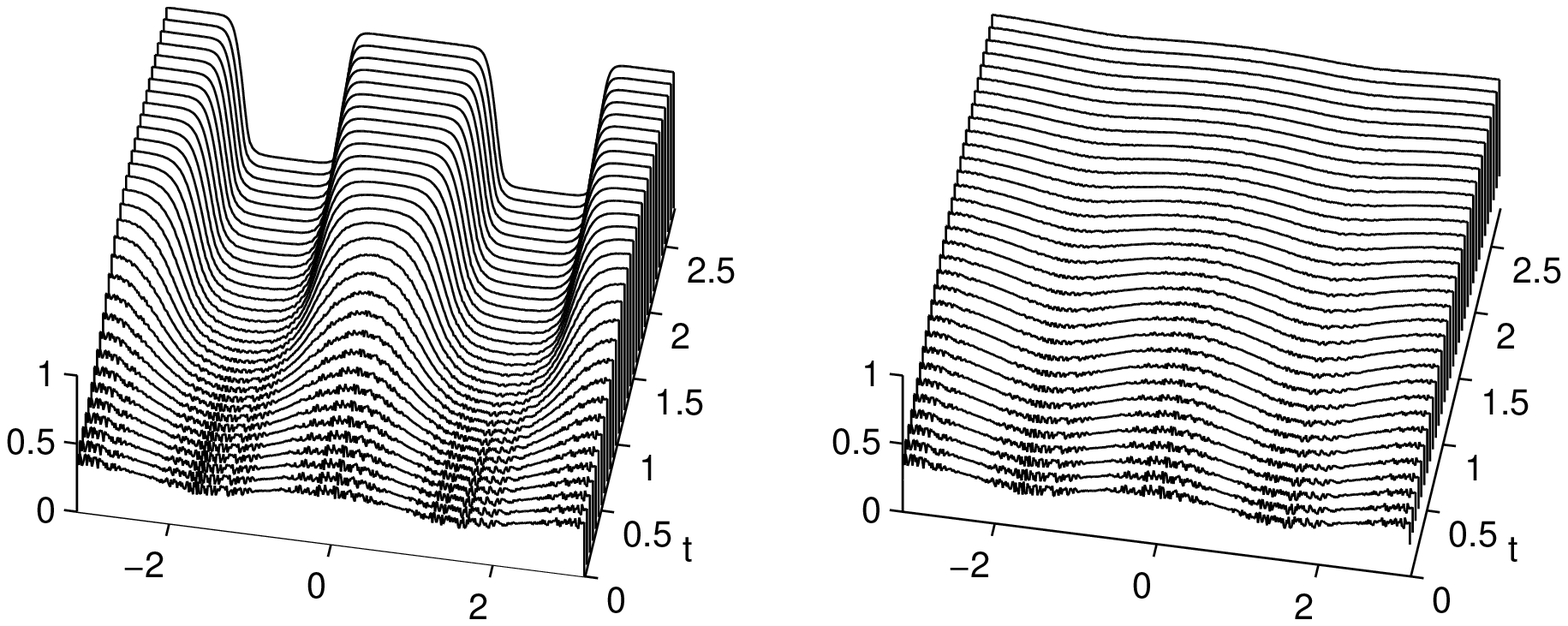}
\caption{\textbf{Replicator versus Logit (Periodic BC).} 
{\protect\footnotesize {The left panel shows population density of strategy
1 for the replicator IDE with $\protect\kappa =\infty $, and the right panel
depicts the population density in the logit dynamics. $\ N=512.$\ $\Lambda
=[-\protect\pi ,\protect\pi ]$\ with the periodic condition$.$\ $%
dt=0.001/(0.05N^{2}),$\ $a_{11}=20/3,$\ $a_{22}=10/3,~a_{12}=a_{21}=0$. \ $%
b=10.$\ for the Gaussian kernel$.~$The initial datum is $\frac{1}{2}+\frac{1%
}{10}\func{rand}\cos (2x),$where rand denotes a realization of the uniform
random variable at each node.}}}
\label{fig:figure11}
\end{figure}
Observe that the time evolution of the replicator dynamic IDE in the left
panel of Figure \ref{fig:figure11} corresponds to the 1-dimensional snap
shot of the pattern formation in two dimensional replicator systems in
Figure \ref{fig-pattern}. In Figure \ref{fig:figure11}, the replicator
system developed a spatial pattern; in the logit dynamic all population
coordinate to an equilibrium of strategy 1 exponentially fast. Thus, in a
society where agents adopt strategies by imitating their neighbors, the
significant proportion of the population may spend a long time in an
inefficient equilibrium, whereas agents with perturbed best response rules
coordinate \textquotedblleft better\textquotedblright\ to an efficient
outcome.

Throughout numerical simulations, we frequently observed the development of
patterns in the replicator IDEs, while this is not the case for logit IDEs,
except for the equal payoff coordination games. We have also observed the
similar pattern formations in the regularized replicator IDEs for a
reasonable range of $\kappa ;$ the regularized replicator IDEs with $\kappa
=10$ showed similar patterns to the replicator IDEs$.$

\subsection{PDE approximations of IDEs}

If the interaction kernel $\mathcal{J}$ is highly concentrated at the
origin, or equivalently, the density $f$ varies slowly with respect to
space, we can consider $\mathcal{J}_{\epsilon }(x)=\epsilon ^{-d}\mathcal{J}%
(x/\epsilon )$\ as an interaction kernel for small $\epsilon .$ Then by a
change of variables and a Taylor expansion, we find%
\begin{equation*}
\mathcal{J}_{\epsilon }\ast f\approx f+\frac{\epsilon ^{2}}{2}J_{2}\Delta f
\end{equation*}%
where we ignore smaller order terms like $\epsilon ^{3}$ and $\Delta
f=(\Delta f_{1},\Delta f_{2},\cdots ,\Delta f_{n})^{T},$ $\Delta f_{1}=\frac{%
\partial ^{2}f_{1}}{\partial r_{1}^{2}}+\cdots +\frac{\partial ^{2}f_{1}}{%
\partial r_{d}^{2}},$ and $J_{2}=\int_{\Lambda }\left\vert w\right\vert ^{2}%
\mathcal{J}(w)dw.$ Thus, by expanding $F(f+\frac{\epsilon ^{2}}{2}%
J_{2}\Delta f,f)$ in equation \ref{IDE1} around $\epsilon $ $\approx 0\,\
\,\,$again, we find the PDE approximations of IDEs:%
\begin{equation}
\frac{\partial f}{\partial t}=F(f,f)+\frac{1}{2}\epsilon ^{2}J_{2}M\Delta f
\label{PDE}
\end{equation}%
where $\left( M\right) _{i,j}:$= $\frac{\partial F_{i}}{\partial r_{j}}$ ,
and the derivatives are evaluated at ($f,f).$ Intuitively, the coordinating
behaviors imply that agents try to choose the same strategy as their
neighbors, and this, in turn, means that the density of a given strategy
tends to diffuse toward locations where the coordination of that strategy is
more likely. This is how our original IDEs are related to the reaction
diffusion equations in (\ref{PDE}). For specific PDE expressions, we find \ 
\begin{equation}
\begin{tabular}{ll}
\textbf{Replicator} & $\frac{\partial f}{\partial t}=\beta f(1-f)(f-\zeta )$
\\ 
& $\ \ \ +\left[ \beta f(1-f)+(1-f)F_{\kappa }\left( \beta \left( f-\zeta
\right) \right) +fF_{\kappa }\left( \beta \left( \zeta -f\right) \right) %
\right] \frac{\epsilon ^{2}}{2}J_{2}\Delta f$ \  \\ 
&  \\ 
\textbf{Logit} & $\frac{\partial f}{\partial t}=l(\beta \left( f-\zeta
\right) )-f+\beta l(\beta (f-\zeta ))(1-l(\beta (f-\zeta )))\frac{\epsilon
^{2}}{2}J_{2}\Delta f$%
\end{tabular}
\label{PDEs}
\end{equation}

Both PDEs in (\ref{PDEs}) are reaction diffusion equations, whose reaction
terms are of the same functional form as the mean field reactions (term $%
\beta f(1-f)(f-\zeta )$ in the replicator and $l(\beta \left( f-\zeta
\right) )-f$ in the logit). The diffusion terms are \emph{nonlinear} as the
coefficients of the terms $\Delta f$ depend on the strategy density $f.$ In
PDE that \citet{Hutson92}, \citet{Vickers93}, \citet{Hofbauer97}, and %
\citet{HofbauerWave97} have studied for the existence of traveling wave
solutions and pattern formation, the diffusion coefficients are constant,
implicitly modeling `fast' diffusion of strategies between players at
different lattice sites in space at the microscopic level, in contrast to
the `slow' strategy updating dynamics; such derivations of
reaction-diffusion PDE from interacting particle systems with combined
fast/slow mechanisms are discussed in \cite{Durrett99} and references
therein. However, in our long-range interaction models the diffusion terms
are concentration-dependent, induced by the nonlinearities in the logit and
replicator microscopic stochastic dynamics, which in turn are heuristically
discussed in Figure \ref{fig:traveling}. In biology models, when the
population pressure tends to enhance dispersal as the population density
increases, the density dependent reaction diffusion models have been used 
\citep{ murray89, Morishita71,
Shigesada80}.

Overall the PDEs in (\ref{PDEs}) provide additional insights for the IDEs in
(\ref{rep-IDE-F}) and (\ref{logit-IDE-F}), and their corresponding
microscopic stochastic dynamics. For example, in the case of (\ref{rep-IDE-F}%
), when $p$ is close to either $0$ or $1$ the diffusion term is weakest and
when $p$ lies in the intermediate range, the effect becomes strong. This
means that the individuals playing strategy 1 diffuse fast, as $p$ reaches $%
\frac{1}{2},$ because it is more likely for them to play with 2-strategists,
so more likely to be uncoordinated. When it is highly likely to be
coordinated, as in $p=0$ or 1, the individuals with the corresponding
strategy do not diffuse at all.

\appendix \setstretch{1} \scalefont{0.8}\newpage

\section{Appendix}

\label{theappendix}

\subsection{Various strategy revision rates and proof of Theorem \protect\ref%
{thm-long-range-fixedBC} \newline
}

\noindent \underline{\textbf{Strategy Revision Rates}} \newline

\noindent We show that condition \textbf{C1}$-$\textbf{C3 }are satisfied for
the following strategy revision rates: \begingroup \scalefont{0.9}

\medskip \noindent $\bullet$ $c^{\gamma }\left( x,\sigma ,k\right) =F\left(
u\left( x,\sigma ,k\right) \right) $ $\ \ \ \ \ \ \ \ \ \ \ \ \ \ \ \ \ \ \
\ \ \ \ \ \ \ \ \ \ \ \ \ \ \ \ \ :$ \ $\mathbf{c(}u,i,k,f)=F\left(
\sum_{l}a\left( i,l\right) \mathcal{J}\ast f\left( u,l\right) \right) $

\medskip \noindent $\bullet$ $c^{\gamma }\left( x,\sigma ,k\right) =F\left(
u\left( x,\sigma ,k\right) -u\left( x,\sigma ,\sigma \left( x\right) \right)
\right) $ $\ \ \ \ \ \ \ \ \ \ \ \ \ \ \ \ \ :$ $\mathbf{c(}u,i,k,f)=F\left(
\sum_{l}[a\left( k,l\right) -a\left( i,l\right) ]\mathcal{J}\ast f\left(
u,l\right) \right) $

\medskip \noindent $\bullet$ $c^{\gamma }\left( x,\sigma ,k\right) =$ $%
\sum_{y}w(x,y,\sigma ,k)F(u(x,\sigma ,k))$ $\ \ \ \ \ \ \ \ \ \ \ \ \ \ \ \
\ \ :$ $\mathbf{c(}u,i,k,f)=\mathcal{J}\ast f\left( u,k\right) F\left(
\sum_{l}a\left( k,l\right) \mathcal{J}\ast f\left( u,l\right) \right) $

\medskip \noindent $\bullet$ $c^{\gamma }\left( x,\sigma ,k\right) =$ $%
\sum_{y}w(x,y,\sigma ,k)F(u\left( x,\sigma ,k\right) -u\left( x,\sigma
,\sigma \left( x\right) \right) )$ $:$ $\mathbf{c(}u,i,k,f)=\mathcal{J}\ast
f\left( u,k\right) F\left( \sum_{l}[a\left( k,l\right) -a\left( i,l\right) ]%
\mathcal{J}\ast f\left( u,l\right) \right) $

\medskip \noindent $\bullet$ $c^{\gamma }\left( x,\sigma ,k\right) =\frac{%
\exp (u\left( x,\sigma ,k\right) )}{\sum_{l}\exp (u\left( x,\sigma ,l\right)
)}$ $\ \ \ \ \ \ \ \ \ \ \ \ \ \ \ \ \ \ \ \ \ \ \ \ \ \ \ \ \ \ \ \ \ \ \ \
:$ $\mathbf{c(}u,i,k,f)=\frac{\exp (\mathcal{J}\ast f\left( u,k\right) )}{%
\sum_{l}\exp (\mathcal{J}\ast f\left( u,l\right) )}$

if $F$ satisfies the global Lipschitz condition:i.e., for all $x,y\in
Dom\left( F\right) ,$ there exists $L>0$ such that $\left\vert
F(x)-F(y)\right\vert \leq L\left\vert x-y\right\vert .$ Note that the list
above is far from being exhaustive; one can easily invent various other
rates which satisfy \textbf{C1}$\mathbf{-}$\textbf{C3}. Since the
verifications of the conditions are similar, we will check the conditions
for the following rate (\ref{check-lem}) in the periodic domain.%
\begin{equation}
c^{\gamma }(x,\sigma ,k)=F(\sum_{y\in \Lambda ^{\gamma }}\mathcal{\gamma }%
^{d}a(k,\sigma (y))\mathcal{J}(\gamma (y-x))-\sum_{y\in \Lambda ^{\gamma }}%
\mathcal{\gamma }^{d}a(\sigma (x),\sigma (y))\mathcal{J}(\gamma (x-y)))
\label{check-lem}
\end{equation}

\begin{lemma}
The rate given by (\ref{check-lem}) satisfies $\mathbf{C1}-\mathbf{C3}.$
\end{lemma}

\begin{proof}
Let%
\begin{eqnarray*}
\bar{c}^{\gamma }(u,i,k,\sigma ) &:&=F(\sum_{y\in \Lambda ^{\gamma }}\gamma
^{d}a(k,\sigma (y))\mathcal{J}(u-\gamma y)-\sum_{y\in \Lambda ^{\gamma
}}\gamma ^{d}a(i,\sigma (y))\mathcal{J}(u-\gamma y)) \\
\mathbf{c}(u,i,k,\pi ) &:&=F(\left\vert \Gamma \right\vert \int^{\gamma
\times S}a(k,l)\mathcal{J}(u-v)\pi (dvdl)-\left\vert \Gamma \right\vert
\int^{\gamma \times S}a(i,l)\mathcal{J}(u-v)\pi (dvdl))
\end{eqnarray*}%
where we associate $u$ to $\gamma x$ and $v$ to $\gamma y.$ First we note
that 
\begin{eqnarray}
&&\left\vert \sum_{y\in \Lambda ^{\gamma }}\gamma ^{d}a(k,\sigma \left(
y\right) )\mathcal{J}(\gamma x-\gamma y)-\left\vert \Gamma \right\vert
\int_{\Lambda \times S}a(k,l)\mathcal{J}(rx-v)\pi ^{\gamma }(\sigma
,dvdl)\right\vert \leq  \label{lem1-1} \\
&&\left\vert \gamma ^{d}-\frac{\left\vert \Lambda \right\vert }{\left\vert
\Lambda ^{\gamma }\right\vert }\right\vert \sum_{y\in \Lambda ^{\gamma
}}a(k,\sigma (y))\mathcal{J}(\gamma x-\gamma y)\leq \left\vert \gamma
^{d}\left\vert \Lambda ^{\gamma }\right\vert -\left\vert \Lambda \right\vert
\right\vert M\rightarrow 0\text{ uniformly in }x,\sigma ,k  \notag
\end{eqnarray}%
where $M:=\sup\limits_{i,k,u,v}a(i,j)\mathcal{J}(u,v).$ So by using the
Lipschitz condition of $F,$ we have%
\begin{eqnarray*}
&&\left\vert c^{\gamma }(x,\sigma ,k)-c(\gamma x,\sigma (x),k,\pi ^{\gamma
}\left( \sigma \right) )\right\vert \leq \left\vert \bar{c}^{\gamma }(\gamma
x,\sigma (x),k,\sigma )-\mathbf{c}(\gamma x,v,\sigma (x),k,\pi ^{\gamma
}(\sigma ))\right\vert \\
&\leq &L\sup_{\substack{ x\in \Lambda ^{\gamma }  \\ \sigma \in S^{\Lambda
^{\gamma }}  \\ k\in S}}\left\vert \sum_{y\in \Gamma ^{\gamma }}\gamma
^{d}a(k,\sigma (y))\mathcal{J}(\gamma x-\gamma y)-\left\vert \Gamma
\right\vert \int_{\Gamma \times S}a(k,l)\mathcal{J}(\gamma x-v)\pi (\sigma
,dvdl)\right\vert + \\
&&L\sup_{\substack{ x\in \Lambda ^{\gamma }  \\ \sigma \in S^{\Lambda
^{\gamma }}  \\ k\in S}}\left\vert \sum_{y\in \Gamma ^{\gamma }}\gamma
^{d}a(\sigma (x),\sigma (y))\mathcal{J}(\gamma x-\gamma y)-\left\vert \Gamma
\right\vert \int_{\Gamma \times S}a(\sigma (x),l)\mathcal{J}(\gamma x-v)\pi
(\sigma ,dvdl)\right\vert \rightarrow 0\text{ uniformly in }x,\sigma ,k
\end{eqnarray*}%
Hence $\mathbf{C1}$ is satisfied. Since $\mathbf{c}(u,i,k,\pi )$ is
uniformly bounded, $\mathbf{C2}$ is satisfied. Again $\mathbf{C3}$ follows
from the fact that $\mathbf{c}\left( u,i,k,\pi \right) $ is uniformly
bounded and $F$ satisfies the Lipschitz condition.
\end{proof}

\bigskip

\noindent \underline{\textbf{Notations }}\newline

\noindent We use the following notations in the proof of Theorem \ref%
{them-long-range-perBC} and Theorem \ref{thm-long-range-fixedBC}.

\medskip \noindent $\bullet$ $\left\{ \Sigma _{t}^{\gamma }\right\} $ is the
stochastic process taking values $\sigma _{t}$ with generator $L^{\gamma }$
given in equation (\ref{eq-generator}) and the sample space $D\left(
[0,T],S^{\Gamma ^{\gamma }}\right) .$\newline

\medskip \noindent $\bullet$ $\left\{ \Pi _{t}^{\gamma }\right\} $ is the
stochastic process for the empirical measure taking values $\pi _{t}$ with
the sample space $D([0,T],\mathcal{P}(\Lambda \times S))$ and we denote by $%
\mathbf{Q}^{\gamma }$ the law of the process $\left\{ \Pi _{t}^{\gamma
}\right\} $ and by $\mathbf{P}$ the probability measure in the underlying
probability space$.$\newline

\noindent The proof of Theorems \ref{them-long-range-perBC} and \ref%
{thm-long-range-fixedBC} are so similar that we only prove Theorem \ref%
{thm-long-range-fixedBC} and leave the modifications needed to prove Theorem %
\ref{them-long-range-perBC} to the readers.

\bigskip

\noindent \underline{\textbf{Martingale Estimates}} \newline

\noindent For $g\in C\left( \Gamma \times S\right) $ we set%
\begin{equation}
h\left( \sigma \right) :=\left\langle \pi ^{\gamma },g\right\rangle =\frac{1%
}{\left\vert \Gamma ^{\gamma }\right\vert }\sum_{y\in \Gamma ^{\gamma
}}g(\gamma y,\sigma (y))  \label{eq:appen_hfun}
\end{equation}%
We define $M_{t}^{g,\gamma },\left\langle M_{t}^{g,\gamma }\right\rangle $
as follows: for $g\in C(\Gamma \times S)$%
\begin{equation}
M_{t}^{g,\gamma }=\left\langle \Pi _{t}^{\gamma },g\right\rangle
-\left\langle \Pi _{0}^{\gamma },g\right\rangle -\int_{0}^{t}L^{\gamma
}\left\langle \Pi _{s}^{\gamma },g\right\rangle ds,\left\langle
M_{t}^{g,\gamma }\right\rangle =\int_{0}^{t}\left[ L^{\gamma }\left\langle
\Pi _{s}^{\gamma },g\right\rangle ^{2}-2\left\langle \Pi _{s}^{\gamma
},g\right\rangle L^{\gamma }\left\langle \Pi _{s}^{\gamma },g\right\rangle %
\right] ds  \label{appen-MG}
\end{equation}%
Since $h$ is measurable, so $M_{t}^{g,\gamma }$ and $\left\langle
M_{t}^{g,\gamma }\right\rangle $ are $\mathcal{F}_{t}-$martingale with
respect to $\mathbf{P},$ where $\mathcal{F}_{t}$ is the filtration generated
by $\left\{ \Sigma _{t}\right\} $ \citep{Ethier86, Darling08}$.$

\begin{lemma}
\label{proof-lem1}For $g\in C(\Gamma \times S)$ there exist $C$ such that%
\begin{equation*}
\left\vert L^{\gamma }\left\langle \pi ^{\gamma },g\right\rangle \right\vert
\leq C,\text{ }\left\vert L^{\gamma }\left\langle \pi ^{\gamma
},g\right\rangle ^{2}-2\left\langle \pi ^{\gamma },g\right\rangle L^{\gamma
}\left\langle \pi ^{\gamma },g\right\rangle \right\vert \leq \gamma ^{d}C
\end{equation*}
\end{lemma}

\begin{proof}
For $h$ in (\ref{eq:appen_hfun})$,$we have%
\begin{equation*}
h(\sigma ^{x,k})-h(\sigma )=\frac{1}{\left\vert \Gamma ^{\gamma }\right\vert 
}\left( g(\gamma x,k)-g(\gamma x,\sigma (x)\right)
\end{equation*}%
and so we have equation (\ref{proof-lem1-1}) below. Now let $q(\sigma
):=\left\langle \pi ^{\gamma },g\right\rangle ^{2}.$ Then%
\begin{eqnarray*}
q(\sigma ^{x,k})-q(\sigma ) &=&\frac{1}{\left\vert \Gamma ^{\gamma
}\right\vert ^{2}}(\sum_{y\in \Lambda ^{\gamma }}g(\gamma y,\sigma
^{x,k}(y)))^{2}-\frac{1}{\left\vert \Gamma ^{\gamma }\right\vert ^{2}}%
(\sum_{y\in \Lambda ^{\gamma }}g(\gamma y,\sigma (y)))^{2} \\
&=&\frac{1}{\left\vert \Gamma ^{\gamma }\right\vert ^{2}}\left( g(\gamma
x,k)-g(\gamma x,\sigma (x))\right) ^{2}+\frac{2}{\left\vert \Gamma ^{\gamma
}\right\vert ^{2}}\left( g(\gamma x,k)-g(\gamma x,\sigma (x))\right)
\sum_{y\in \Lambda ^{\gamma }}g(\gamma y,\sigma (y))
\end{eqnarray*}%
Thus we have%
\begin{equation}
L^{\gamma }\left\langle \pi ^{\gamma },g\right\rangle =\frac{1}{\left\vert
\Gamma ^{\gamma }\right\vert }\sum_{k\in S}\sum_{x\in \Lambda ^{\gamma
}}c^{\gamma }(x,\sigma (x),k)\left( g(\gamma x,k)-g(\gamma x,\sigma
(x)\right)  \label{proof-lem1-1}
\end{equation}%
\begin{equation}
L^{\gamma }\left\langle \pi ^{\gamma },g\right\rangle ^{2}-2\left\langle \pi
^{\gamma },g\right\rangle L^{\gamma }\left\langle \pi ^{\gamma
},g\right\rangle =\frac{1}{\left\vert \Gamma ^{\gamma }\right\vert ^{2}}%
\sum_{k\in S}\sum_{x\in \Lambda ^{\gamma }}c^{\gamma }(x,\sigma (x),k)\left(
g(\gamma x,k)-g(\gamma x,\sigma (x)\right) ^{2}  \label{proof-lem1-2}
\end{equation}%
Therefore from $\mathbf{C1}-\mathbf{C2},$ $\left\vert \Gamma ^{\gamma
}\right\vert \approx \left\vert \Gamma \right\vert \gamma ^{-d},$and $%
\left\vert \Lambda ^{\gamma }\right\vert \approx \left\vert \Lambda
\right\vert \gamma ^{-d},$ the results follow.
\end{proof}

\begin{proposition}
\label{proof-lem-e2}Let $g\in C\left( \Gamma \times S\right) $ and $\tau
^{\gamma }$ and $\delta ^{\gamma }$ such that \newline
(1) $\tau ^{\gamma }$ is a stopping time on the process $\left\{ \Pi
_{t}^{\gamma }:0\leq t\leq T\right\} $ with respect to the filtration $%
\mathcal{F}_{t}$.\newline
(2) $\delta ^{\gamma }$ is a constant, $0\leq \delta ^{\gamma }\leq T$ and $%
\delta ^{\gamma }\rightarrow 0$ as $\gamma \rightarrow 0.$\newline
Then for $\epsilon >0,$ there exists $C$ such that%
\begin{equation*}
\text{(i) \ }\mathbf{P}\left\{ \omega :\sup_{t\in \lbrack 0,T]}\left\vert
M_{t}^{g,\gamma }\right\vert \geq \epsilon \right\} \leq \frac{\gamma ^{d}CT%
}{\epsilon ^{2}}\text{ \ and (ii) \ }\mathbf{P}\left\{ \omega :\left\vert
M_{\tau ^{\gamma }+\delta ^{\gamma }}^{g,\gamma }-M_{\tau ^{\gamma
}}^{g,\gamma }\right\vert \geq \epsilon \right\} \leq \frac{\gamma
^{d}C\delta ^{\gamma }}{\epsilon ^{2}}
\end{equation*}%
and there exists $\gamma _{0}$ such that for $\gamma <\gamma _{0}$%
\begin{equation*}
\text{(iii)~\ }\mathbf{P}\left\{ \omega :\left\vert \int_{\tau ^{\gamma
}}^{\tau ^{\gamma }+\delta ^{\gamma }}L^{\gamma }\left\langle \Pi
_{s}^{\gamma },g\right\rangle ds\right\vert \geq \epsilon \right\} =0
\end{equation*}
\end{proposition}

\begin{proof}
We first show (iii). Let $C$ as in Lemma \ref{proof-lem1}. Since $\delta
^{\gamma }\rightarrow 0,$ there exists $\gamma _{0}$ such that $\delta
^{\gamma }<\frac{\epsilon }{2C}$ for $\gamma \leq \gamma _{0}.$ Then by
Lemma \ref{proof-lem1} 
\begin{equation*}
\left\vert \int_{\tau ^{\gamma }}^{\tau ^{\gamma }+\delta ^{\gamma
}}L^{\gamma }\left\langle \Pi _{s}^{\gamma },g\right\rangle ds\right\vert
\leq \delta ^{\gamma }C<\frac{\epsilon }{2},\text{ for }\gamma \leq \gamma
_{0}.
\end{equation*}
For (i), let $\gamma $ be fixed first. Since\ $(M_{0}^{g,\gamma
})^{2}-\left\langle M_{0}^{g,\gamma }\right\rangle =$ $0,\mathbf{P}$ a.e.
and $(M_{t}^{g,\gamma })^{2}-\left\langle M_{t}^{g,\gamma }\right\rangle $
is $\mathcal{F}_{t}-$martingale, by martingale inequality and Lemma \ref%
{proof-lem1}, we have, 
\begin{equation*}
\mathbf{P}\left\{ \omega :\sup_{t\in \lbrack 0,T]}\left\vert M_{t}^{g,\gamma
}\right\vert >\epsilon \right\} \leq \frac{1}{\epsilon ^{2}}\mathbf{E}\left[
\left( M_{T}^{g,\gamma }\right) ^{2}\right] \text{\ \ =}\frac{1}{\epsilon
^{2}}\mathbf{E}\left[ \left\langle M_{T}^{g,\gamma }\right\rangle \right]
\leq \frac{\gamma ^{d}CT}{\epsilon ^{2}}
\end{equation*}%
For (ii), by Lemma \ref{proof-lem1}, Chevyshev inequality, and Doob's
optional stopping, we have 
\begin{equation*}
\mathbf{P}\left\{ \omega :\left\vert M_{\tau ^{\gamma }+\delta ^{\gamma
}}^{g,\gamma }-M_{\tau ^{\gamma }}^{g,\gamma }\right\vert >\epsilon \right\}
\leq \frac{1}{\epsilon ^{2}}\mathbf{E}\left[ (M_{\tau ^{\gamma }+\delta
^{\gamma }}^{g,\gamma }-M_{\tau ^{\gamma }}^{g,\gamma })^{2}\right] =\frac{1%
}{\epsilon ^{2}}\mathbf{E}\left[ \left\langle M_{\tau ^{\gamma }+\delta
^{\gamma }}^{g,\gamma }\right\rangle -\left\langle M_{\tau ^{\gamma
}}^{g,\gamma }\right\rangle \right] \leq \frac{\gamma ^{d}C\delta ^{\gamma }%
}{\epsilon ^{2}}
\end{equation*}
\end{proof}

Next we prove an exponential estimate. We let $r_{\theta }(x)=e^{\theta
\left\vert x\right\vert }-1-\theta \left\vert x\right\vert $ and $s_{\theta
}(x)=e^{\theta x}-1-\theta x$ for $x,$ $\theta \in \mathbb{R}.$ We define%
\begin{equation*}
\text{\ }\phi (\sigma ,\theta ):=\sum_{k\in S}\sum_{x\in \Lambda ^{\gamma
}}c^{\gamma }(x,\sigma ,k)r_{\theta }(h(\sigma ^{x,k})-h\left( \sigma
\right) ),\text{ }\psi (\sigma ,\theta ):=\sum_{k\in S}\sum_{x\in \Lambda
^{\gamma }}c^{\gamma }(x,\sigma ,k)s_{\theta }(h(\sigma ^{x,k})-h\left(
\sigma \right) )\text{ }
\end{equation*}%
Then, from Proposition 8.8 in \citet{Darling08}, we have for $%
M_{T}^{g,\gamma }$ in (\ref{appen-MG})%
\begin{equation*}
Z_{t}^{g,\gamma }:=\exp \left\{ \theta M_{t}^{g,\gamma }-\int_{0}^{t}\psi
(\Sigma _{s}^{\gamma },\theta )ds\right\}
\end{equation*}%
is a supermartingale for $\theta \in \mathbb{R}$. Now we let $C_{g}:=2\sup
\left\vert g(u,i)\right\vert ,$ \ $C_{c}:=\sup \left\vert c^{\gamma
}(x,\sigma ,k)\right\vert .$

\begin{lemma}[Exponential Estimate]
\label{lem-exp-est}There exist $C$ depends on $C_{g},C_{c},S$ and $\epsilon
_{0}$ such that for all $\epsilon \leq \epsilon _{0}$ we have%
\begin{equation*}
\mathbf{P}\left\{ \sup_{t\leq T}\left\vert M_{t}^{g,\gamma }\right\vert \geq
\epsilon \right\} \leq 2e^{-\frac{\left\vert \Lambda ^{\gamma }\right\vert 
\text{ }\epsilon ^{2}}{TC}}
\end{equation*}
\end{lemma}

\begin{proof}
We choose $\epsilon _{0}\leq \frac{1}{2}\left\vert S\right\vert C_{g}C_{c}T$
and let $A=\frac{1}{\left\vert \Lambda ^{\gamma }\right\vert }\left\vert
S\right\vert C_{g}^{2}C_{c}e,$ $\theta =\frac{\epsilon }{AT}.$ Then since $%
r_{\theta }$ is increasing in $\mathbb{R}_{+},$ 
\begin{equation*}
r_{\theta }\left( h(\sigma ^{x,k})-h(\sigma )\right) \leq r_{\theta }\left( 
\frac{1}{\left\vert \Lambda ^{\gamma }\right\vert }C_{g}\right) \leq \frac{1%
}{2}\left( \frac{1}{\left\vert \Lambda ^{\gamma }\right\vert }C_{g}\theta
\right) ^{2}e^{\frac{1}{\left\vert \Lambda ^{\gamma }\right\vert }\theta
C_{g}}\text{ for all }\sigma \in S^{\Lambda \gamma }
\end{equation*}%
where in the last line we used $e^{x}-1-x\leq \frac{1}{2}x^{2}e^{x}$ for all 
$x>0.$ Also for $\epsilon \leq \epsilon _{0},$%
\begin{equation*}
\frac{1}{\left\vert \Lambda ^{\gamma }\right\vert }\theta C_{g}=\frac{1}{%
\left\vert \Lambda ^{\gamma }\right\vert }\frac{\epsilon }{AT}C_{g}\leq 
\frac{1}{\left\vert \Lambda ^{\gamma }\right\vert }\frac{1}{2}\frac{%
\left\vert S\right\vert C_{g}^{2}C_{c}}{A}\leq \frac{1}{2e}<1
\end{equation*}%
Thus%
\begin{equation*}
\int_{0}^{T}\phi (\Sigma _{t}^{\gamma },\theta )dt\leq \left\vert
S\right\vert \left\vert \Lambda ^{\gamma }\right\vert \frac{1}{\left\vert
\Lambda ^{\gamma }\right\vert ^{2}}\frac{1}{2}C_{g}^{2}\theta ^{2}e^{\frac{1%
}{\left\vert \Lambda ^{\gamma }\right\vert }\theta C_{g}}C_{c}T\leq \frac{1}{%
2}\frac{1}{\left\vert \Lambda ^{\gamma }\right\vert }\left\vert S\right\vert
C_{g}^{2}C_{c}e\theta ^{2}T=\frac{1}{2}A\theta ^{2}T\text{ \ for all }\omega
\in \Omega
\end{equation*}%
So, since $\psi (\sigma ,\theta )\leq $ $\phi (\sigma ,\theta ),$ 
\begin{eqnarray*}
\mathbf{P}\left\{ \sup_{t\leq T}M_{t}^{g,\gamma }>\epsilon \right\} &=&%
\mathbf{P}\left\{ \sup_{t\leq T}Z_{t}^{g,\gamma }>\exp [\theta \epsilon
-\int_{0}^{T}\psi (\Sigma _{t}^{\gamma },\theta )dt]\right\} \leq \mathbf{P}%
\left\{ \sup_{t\leq T}Z_{t}^{g,\gamma }>\exp [\theta \epsilon -\frac{1}{2}%
A\theta ^{2}T]\right\} \\
&\leq &e^{\frac{1}{2}A\theta ^{2}T-\theta \epsilon }=e^{-\frac{\left\vert
\Lambda ^{\gamma }\right\vert \text{ }\epsilon ^{2}}{TC}}
\end{eqnarray*}%
where we choose $C:=2\left\vert S\right\vert C_{g}^{2}C_{c}e.$ Since the
same inequality holds for $-M_{t}^{g,\gamma },$ we obtain the desired result.%
\newline
\newline
\end{proof}

\noindent \underline{\textbf{Convergence}} \newline

\begin{lemma}[Relative Compactness]
The sequence $\left\{ \mathbf{Q}^{\gamma }\right\} $ in $\mathcal{P}\left(
D\left( [0,T];\mathcal{P}(\Lambda \times S)\right) \right) $ is relatively
compact.
\end{lemma}

\begin{proof}
By Proposition 1.7 in \citet[][p.54]{Kipnis99}, we show that $\left\{ 
\mathbf{Q}^{\gamma }g^{-1}\right\} $ is relatively compact in $\mathcal{P}%
\left( D([0,T];\mathbb{R)}\right) $ for each $g\in C(\Lambda \times S),$
where the definition of $\mathbf{Q}^{\gamma }g^{-1}$ is as follows$:$ for
any Borel set $A$ in $D([0,T];\mathbb{R)}$%
\begin{equation*}
\mathbf{Q}^{\gamma }g^{-1}(A):=\mathbf{Q}^{\gamma }\left\{ \pi .\in D\left(
[0,T];\mathcal{P}(\Lambda \times S)\right) :\left\langle \pi
.,g\right\rangle \in A\right\}
\end{equation*}%
So, from Theorem 1 in \citet{Aldous78} and Prohorov Theorem in %
\citet[][p.125]{Billingsley68}, it is enough to show that \newline
(i) for $\eta >0,$ there exists $a$ such that%
\begin{equation*}
\mathbf{Q}^{\gamma }g^{-1}\left\{ x\in D\left( [0,T];\mathbb{R}\right)
:\sup_{t}\left\vert x(t)\right\vert >a\right\} \leq \eta \text{ for }\gamma
\leq 1
\end{equation*}%
\newline
(ii)%
\begin{equation*}
\mathbf{P}\left\{ \omega :\left\vert \left\langle \Pi _{\tau ^{\gamma
}+\delta ^{\gamma }}^{\gamma },g\right\rangle -\left\langle \Pi _{\tau
^{\gamma }}^{\gamma },g\right\rangle \right\vert >\epsilon \right\}
\rightarrow 0
\end{equation*}%
for all $\epsilon >0,$ for $(\tau ^{\gamma },\delta ^{\gamma })$ satisfying
the condition (1) and (2) of Proposition \ref{proof-lem-e2}$.$ For (i),
since $g$ is bounded, it is enough to choose $a=2\sup \left\vert
g(u,i)\right\vert $; i.e., $\mathbf{Q}^{\gamma }g^{-1}\left\{ x\in D\left(
[0,T];\mathbb{R}\right) :\sup_{t}\left\vert x(t)\right\vert >a\right\} =%
\mathbf{Q}^{\gamma }\{\pi .:\sup_{t}\left\vert \left\langle \pi
_{t},g\right\rangle \right\vert >a\}=0$ since $\left\vert \left\langle \pi
.,g\right\rangle \right\vert <a$ for all $\pi .$ For (ii) 
\begin{eqnarray*}
&&\mathbf{P}\left\{ \omega :\left\vert \left\langle \Pi _{\tau ^{\gamma
}+\delta ^{\gamma }}^{\gamma },g\right\rangle -\left\langle \Pi _{\tau
^{\gamma }}^{\gamma },g\right\rangle \right\vert >\epsilon \right\} \leq 
\mathbf{P}\left\{ \omega :\left\vert M_{\tau ^{\gamma }+\delta ^{\gamma
}}^{g,\gamma }-M_{\tau ^{\gamma }}^{g,\gamma }\right\vert >\frac{\epsilon }{2%
}\right\} +\mathbf{P}\left\{ \omega :\sup_{t\in \lbrack 0,T]}\left\vert
M_{t}^{g,\gamma }\right\vert >\frac{\epsilon }{2}\right\} \\
&\leq &\frac{\gamma ^{d}C\delta ^{\gamma }}{\epsilon ^{2}}\text{ \ \ for }%
\gamma \leq \gamma _{0}\text{ chosen in Proposition \ref{proof-lem-e2}}
\end{eqnarray*}
\end{proof}

Let $\mathbf{Q}^{\ast }$ be a limit point of $\left\{ \mathbf{Q}^{\gamma
}\right\} $ and choose a subsequence $\left\{ \mathbf{Q}^{\gamma
_{k}}\right\} $ converging weakly to $\mathbf{Q}^{\ast }.$ Hereafter we
denote the stochastic process defined on $\Lambda ^{\gamma }$ by $\left\{
\Sigma ^{\Lambda ^{\gamma }}\right\} $ and its restriction on $\Gamma
^{\gamma }$ by $\left\{ \Sigma ^{\Gamma ^{\gamma }}\right\} .$ With these
notations, equation (\ref{appen-MG}) becomes%
\begin{equation}
\left\langle \Pi _{t}^{\Gamma ^{\gamma }},g\right\rangle =\left\langle \Pi
_{0}^{\Gamma ^{\gamma }},g\right\rangle +\frac{\left\vert \Lambda ^{\gamma
}\right\vert }{\left\vert \Gamma ^{\gamma }\right\vert }\int_{0}^{t}ds%
\sum_{k\in S}\int_{\mathbf{\Lambda }\times S}c(u,i,k,\Pi _{s}^{\Gamma
^{\gamma }})\left( g(u,k)-g(u,i)\right) \Pi _{s}^{\Lambda ^{\gamma
}}(dudi)+M_{t}^{g,\gamma }  \label{eq:mg-fixed}
\end{equation}%
Let $\pi \in \mathcal{P}(\Gamma \times S)$ and we define $\pi _{\Lambda
}(du,di):=\pi (du\cap \Lambda ,di)$.

\begin{lemma}[Characterization of Limit Points]
For all $\epsilon >0,$%
\begin{equation*}
\mathbf{Q}^{\ast }\left\{ \pi .:\sup_{t\in \lbrack 0,T]}\left\langle \pi
_{t},g\right\rangle -\left\langle \pi _{0},g\right\rangle
-\int_{0}^{t}ds\sum_{k\in S}\left[ \int_{\Lambda \times S}c(u,i,k,\pi
_{s})(g(u,k)-g(u,i))\pi _{\Lambda ,s}(dudi)\right] >\epsilon \right\} =0,
\end{equation*}%
i.e. the limiting process is concentrated on the weak solutions of the IDE (%
\ref{lf-de}).
\end{lemma}

\begin{proof}
First we define $\Phi :D(\left[ 0,T\right] ,\mathcal{P}(\Lambda \times
S))\rightarrow \mathbb{R}$ 
\begin{equation*}
\pi .\mapsto \left\vert \sup_{t\in \lbrack 0,T]}\left\langle \pi
_{t},g\right\rangle -\left\langle \pi _{0},g\right\rangle
-\int_{0}^{t}ds\sum_{k\in S}\left[ \int_{\Lambda \times S}c(u,i,k,\pi
_{s})(g(u,k)-g(u,i))\pi _{\Lambda ,s}(dudi)\right] \right\vert
\end{equation*}%
Then $\Phi $ is continuous, hence $\Phi ^{-1}(\left( \epsilon ,\infty
\right) )$ is open. From the weak convergence of $\left\{ \mathbf{Q}^{\gamma
_{k}}\right\} $ to $\mathbf{Q}^{\ast }$,%
\begin{equation*}
\mathbf{Q}^{\ast }\left\{ \pi .:\Phi (\pi .)>\epsilon \right\} \leq
\liminf_{l\rightarrow \infty }\mathbf{Q}^{\gamma _{l}}\left\{ \pi .:\Phi
(\pi .)>\epsilon \right\}
\end{equation*}%
Also,%
\begin{equation*}
\mathbf{Q}^{\gamma }\left\{ \pi .:\Phi (\pi .)>\epsilon \right\} =\mathbf{P}%
\left\{ \omega :\sup_{t\in \lbrack 0,T]}\left\vert M_{t}^{g,\gamma
}\right\vert >\epsilon \right\} \leq \frac{\gamma ^{d}CT}{\epsilon ^{2}}%
\text{ (by Proposition \ref{proof-lem-e2}) for }\gamma <\gamma _{0}
\end{equation*}%
The first equality follows from (\ref{eq:mg-fixed}) and the following
equality:%
\begin{equation*}
\Pi _{\Lambda ,s}(dudi)=\frac{1}{\left\vert \Gamma ^{\gamma }\right\vert }%
\sum_{x\in \Gamma ^{\gamma }\cap \Lambda }\delta _{(\gamma x,\Sigma
_{s}^{\Gamma ^{\gamma }}(x))}(dudi)=\frac{1}{\left\vert \Gamma ^{\gamma
}\right\vert }\sum_{x\in \Lambda ^{\gamma }}\delta _{(\gamma x,\Sigma
_{s}^{_{\Lambda \gamma }}(x))}(dudi)=\frac{\left\vert \Lambda ^{\gamma
}\right\vert }{\left\vert \Gamma ^{\gamma }\right\vert }\Pi _{s}^{\Lambda
^{\gamma }}(dudi).
\end{equation*}
\end{proof}

We denote explicitly by $dm\otimes dv$ as a product measure of Lebesgue
measure on $\Lambda $ and the counting measure on $S.$

\begin{lemma}[Absolutely Continuity]
We have%
\begin{equation*}
\mathbf{Q}^{\ast }\{\pi .:\pi _{t}(dudi)\text{ is absolutely continuous }%
\text{with respect to }dm\otimes dv\text{ for all }t\in \lbrack 0,T]\}=1\,.
\end{equation*}
\end{lemma}

\begin{proof}
We define $\Phi :D(\left[ 0,T\right] ;\mathcal{P}(\Gamma \times
S))\rightarrow \mathbb{R},\pi .\mapsto \sup_{t\in \lbrack 0,T]}\left\vert
\left\langle \pi _{t},g\right\rangle \right\vert .$ Then $\Phi $ is
continuous. Also%
\begin{equation*}
\left\vert \left\langle \pi ^{\gamma },g\right\rangle \right\vert \leq \frac{%
1}{\left\vert \Gamma ^{\gamma }\right\vert }\sum_{x\in \Gamma ^{\gamma
}}\left\vert g(\gamma x,\sigma (x))\right\vert \leq \sum_{l\in S}\frac{1}{%
\left\vert \Gamma ^{\gamma }\right\vert }\sum_{x\in \Gamma ^{\gamma
}}\left\vert g(\gamma x,l)\right\vert
\end{equation*}%
Thus%
\begin{equation*}
\sup_{t\in \lbrack 0,T]}\left\vert \left\langle \pi _{t}^{\gamma
},g\right\rangle \right\vert \leq \sum_{l\in S}\frac{1}{\left\vert \Gamma
^{\gamma }\right\vert }\sum_{x\in \Gamma ^{\gamma }}\left\vert g(\gamma
x,l)\right\vert
\end{equation*}%
We write $\pi _{\cdot }^{\ast }$ be a trajectory on which all $\mathbf{Q}%
^{\ast }$'s are concentrated. Then $\Pi _{\cdot }^{\gamma }\overset{\mathcal{%
D}}{\longrightarrow }\pi _{\cdot }^{\ast }$ (convergence in distribution),
so $\mathbf{E}\left( \Phi (\Pi _{\cdot }^{\gamma })\right) \rightarrow 
\mathbf{E}\left( \Phi (\pi _{\cdot }^{\ast })\right) .$ Also $\frac{1}{%
\left\vert \Gamma ^{\gamma }\right\vert }\sum_{x\in \Gamma ^{\gamma
}}\left\vert g(\gamma x,l)\right\vert \rightarrow \int_{\Lambda }\left\vert
g(u,l)\right\vert du$ for all $l$ by the Riemann sum approximation. Thus,%
\begin{equation*}
\sup_{t\in \lbrack 0,T]}\left\vert \left\langle \pi _{t}^{\ast
},g\right\rangle \right\vert =\Phi (\pi _{\cdot }^{\ast })=\lim_{\gamma
\rightarrow 0}\mathbf{E}\left( \Phi (\Pi _{\cdot }^{\gamma })\right) \leq
\lim_{\gamma \rightarrow 0}\sum_{l\in S}\frac{1}{\left\vert \Gamma ^{\gamma
}\right\vert }\sum_{x\in \Gamma ^{\gamma }}\left\vert g(\gamma
x,l)\right\vert =\int_{\mathbf{\Gamma }\times S}\left\vert g\left(
u,l\right) \right\vert dm\otimes dv
\end{equation*}%
Therefore, for all $t\in \lbrack 0,T],$ for all $g\in C(\Gamma \times S)$, 
\begin{equation*}
\left\vert \int_{\mathbf{\Gamma }\times S}g(u,l)\pi _{t}^{\ast
}(dudl)\right\vert \leq \int_{\mathbf{\Gamma }\times S}\left\vert g\left(
u,l\right) \right\vert dm\otimes dv
\end{equation*}%
so for all $t\in \lbrack 0,T]$ $\pi _{t}^{\ast }$ is absolutely continuous
with respect to $dm\otimes dv.$
\end{proof}

We also see that all limit points of the sequence $\left\{ \mathbf{Q}%
^{\gamma }\right\} $ are concentrated on the trajectories that equal to $%
f^{0}dudi$ at time 0$,$ since 
\begin{eqnarray*}
&&\mathbf{Q}^{\ast }\left\{ \pi .:\left\vert \int g(u,i)\pi _{0}(dudi)-\frac{%
1}{\left\vert \Gamma \right\vert }\int g(u,i)f^{0}(u,i)dudi\right\vert
>\epsilon \right\} \\
&\leq &\liminf_{k\rightarrow \infty }\mathbf{Q}^{\gamma _{k}}\left\{ \pi
.:\left\vert \int g(u,i)\pi _{0}(dudi)-\frac{1}{\left\vert \Gamma
\right\vert }\int g(u,i)f^{0}(u,i)dudi\right\vert >\epsilon \right\} =0,
\end{eqnarray*}%
where the definition of sequence of product measures with a slowly varying
parameter implies the last equality by Proposition 0.4 %
\citet[][p.44]{Kipnis99}.

So far we have shown that $\mathbf{Q}^{\ast }$'s are concentrated on the
trajectories that are the weak solutions of the integro-differential
equations. Next we show the uniqueness of weak solutions defined in the
following way. Let $\mathcal{A}(f)(u,i):=\sum_{k\in S}\mathbf{c}%
(u,k,i,f)f_{\Lambda }(t,u,k)-f_{\Lambda }(t,u,i)\sum_{k\in S}\mathbf{c}%
(u,i,k,f).$ For an initial profile $f^{0}\in \mathcal{M},$\ $f\in \mathcal{M}
$\ is a weak solution of the Cauchy problem: 
\begin{equation}
\frac{\partial f_{t}}{\partial t}=\mathcal{A}(f_{t}),\text{ }f_{0}=f^{0}
\label{proof-cauchy}
\end{equation}%
if for every function $g\in C(\Gamma \times S),$ for all $t<T,\left\langle
f_{t},g\right\rangle =\int_{0}^{t}\left\langle \mathcal{A}%
(f_{s}),g\right\rangle ds.$Observe that from $\mathbf{C3}$ $\mathcal{A}$
satisfies the Lipschitz condition: there exists $C$ such that for all $f,%
\tilde{f}\in L^{\infty }\left( [0,T];\text{ }L^{\infty }(\Gamma \times
S\right) ),$ $\left\Vert \mathcal{A}(f)-\mathcal{A}(\tilde{f})\right\Vert
_{L^{2}(\Gamma \times S)}\leq C\left\Vert f-\tilde{f}\right\Vert
_{L^{2}(\Gamma \times S)}.$

\begin{lemma}[Uniqueness of Weak Solutions]
Weak solutions of the Cauchy problem (\ref{proof-cauchy}) which belong to 
\newline
$L^{\infty }\left( [0,T];L^{2}(\Gamma \times S)\right) $ are unique.
\end{lemma}

\begin{proof}
Let $f_{t},$ $\tilde{f}_{t}$ be two weak solutions and $\bar{f}_{t}:=f_{t}-%
\tilde{f}_{t}$. Then, we have%
\begin{equation*}
\left\langle \bar{f}_{t},g\right\rangle =\int_{0}^{t}\left\langle \mathcal{A(%
}f_{s})-\mathcal{A(}\tilde{f}_{s}),g\right\rangle ds\text{ for all }g\in
C(\Gamma \times S)
\end{equation*}%
We show that $t\mapsto \left\Vert \bar{f}_{t}\right\Vert _{L^{2}(\Gamma
\times S)}^{2}$ is differentiable. Define a mollifier $\eta (x):=C\exp
\left( \frac{1}{\left\vert x\right\vert -1}\right) $ if $\left\vert
x\right\vert <1,$ \ $:=$ $0\ $if $\left\vert x\right\vert \geq 1,$ $C>0$ is
a constant such that $\int_{\mathbb{R}^{d}}\eta (x)dx=1.$ For $\epsilon >0,$
set $\eta _{\epsilon }(x):=\epsilon ^{-d}\eta (\epsilon ^{-1}x).$ For each $%
u\in \Gamma ,$ $i\in S$, define $h_{u,i}^{\epsilon }(v,k)=\eta _{\epsilon
}\left( u-v\right) \mathbf{1}_{\left\{ i=k\right\} }$ and%
\begin{equation*}
\bar{f}_{t}^{\epsilon }(u,i):=\int_{\Gamma \times S}\left( f_{t}(v,k)-\tilde{%
f}_{t}(v,k)\right) h_{u,i}^{\epsilon }(v,k)dvdk
\end{equation*}%
Then, 
\begin{equation*}
\left\vert \left\langle \mathcal{A(}f_{s})-\mathcal{A(}\tilde{f}%
_{s}),h_{u,i}^{\epsilon }\right\rangle \right\vert \leq \left\Vert \mathcal{%
A(}f_{s})-\mathcal{A(}\tilde{f}_{s})\right\Vert _{L^{2}}\left\Vert
h_{u,i}^{\epsilon }\right\Vert _{L^{2}}\leq C\left\Vert f_{s}-\tilde{f}%
_{s}\right\Vert _{L^{2}}\left\Vert h_{u,i}^{\epsilon }\right\Vert
_{L^{2}}\leq C\sup_{s\in \lbrack 0,T]}\left\Vert f_{s}-\tilde{f}%
_{s}\right\Vert _{L^{2}}\left\Vert h_{u,i}^{\epsilon }\right\Vert _{L^{2}}.
\end{equation*}%
Since $f_{s}-\tilde{f}_{s}\in $ $L^{\infty }\left( [0,T];L^{2}(\Gamma \times
S)\right) $ and $h_{u,i}^{\epsilon }\in C(\Gamma \times S)$ for each $u,i,$ $%
t\mapsto \bar{f}_{t}^{\epsilon }(u,i)$ is differentiable and its derivative
is $\bar{f}_{t}^{\epsilon \prime }(u,i)=\left\langle \mathcal{A(}f_{s})-%
\mathcal{A(}\tilde{f}_{s}),h_{u,i}^{\epsilon }\right\rangle .$\ Also, it
follows that $\left\Vert \bar{f}_{t}^{\epsilon }\right\Vert _{L^{2}}^{2}$ is
differentiable with respect to $t$ and%
\begin{equation*}
\frac{d}{dt}\left\Vert \bar{f}_{t}^{\epsilon }\right\Vert
_{L^{2}}^{2}=\int_{\Gamma \times S}2\left\langle \mathcal{A(}f_{t})-\mathcal{%
A(}\tilde{f}_{t}),h_{u,i}^{\epsilon }\right\rangle \bar{f}_{t}^{\epsilon
}(u,i)dudi,\text{ so\ }\left\Vert \bar{f}_{t}^{\epsilon }\right\Vert
_{L^{2}}^{2}=\int_{0}^{t}\left[ \int_{\Gamma \times S}2\left\langle \mathcal{%
A(}f_{s})-\mathcal{A(}\tilde{f}_{s}),h_{u,i}^{\epsilon }\right\rangle \bar{f}%
_{t}^{\epsilon }(u,i)dudi\right] ds\text{ }
\end{equation*}%
Then since $f_{t}^{\epsilon }\rightarrow f_{t}$ in$~\left\Vert \text{ }%
\right\Vert _{L^{2}}$ and $\bar{f}_{t}\in L^{\infty }\left(
[0,T];L^{2}(\Gamma \times S)\right) $ for a given $\ t,$ we have \TEXTsymbol{%
\vert} $\left\Vert \bar{f}_{t}^{\epsilon }\right\Vert
_{L^{2}}^{2}-\left\Vert \bar{f}_{t}\right\Vert _{L^{2}}^{2}|$ $\rightarrow
0. $ Also because $\left\langle \mathcal{A(}f_{s})-\mathcal{A(}\tilde{f}%
_{s}),h_{u,i}^{\epsilon }\right\rangle \rightarrow \mathcal{A}(f_{t})(u,i)-%
\mathcal{A}(\tilde{f}_{t})(u,i)$ for $a.e.u,$ and all $i,t$, by the dominant
convergence theorem we have%
\begin{equation*}
\left\Vert \bar{f}_{t}\right\Vert _{L^{2}}^{2}=\int_{0}^{t}2\left\langle 
\mathcal{A}(f_{s})-\mathcal{A}(\tilde{f}_{s}),\bar{f}_{s}\right\rangle ds,
\end{equation*}%
so $\left\Vert \bar{f}_{t}\right\Vert _{L^{2}}^{2}$ is differentiable and 
\begin{equation*}
\frac{d}{dt}\left\Vert \bar{f}_{t}\right\Vert _{L^{2}}^{2}=\left\langle 
\mathcal{A}(f_{t})-\mathcal{A}(\tilde{f}_{t}),\bar{f}_{t}\right\rangle \leq
2\left\Vert \mathcal{A}(f_{t})-\mathcal{A}(\tilde{f}_{t})\right\Vert
_{L^{2}}\left\Vert \bar{f}_{t}\right\Vert _{L^{2}}\leq C\left\Vert \bar{f}%
_{t}\right\Vert _{L^{2}}^{2}
\end{equation*}%
Hence from Gronwell lemma, the uniqueness of the solutions follows.
\end{proof}

\begin{lemma}[Convergence in Probability]
\label{lem-conv}We have 
\begin{equation*}
\Pi _{t}^{\gamma }(du,di)\longrightarrow \frac{1}{\left\vert \Gamma
\right\vert }f(t,u,i)\,dudi\text{ in probability }\,.
\end{equation*}
\end{lemma}

\begin{proof}
So far we established $\mathbf{Q}^{\gamma }\Rightarrow \mathbf{Q}^{\ast }$
(converge weakly) and equivalently $\Pi _{\cdot }^{\gamma }\rightarrow \pi
_{\cdot }^{\ast }$ in Skorohod topology (topology on $D([0,T],\mathcal{P}(%
\mathbf{T}^{d}\times S))$). If we show that $\Pi _{t}^{\gamma }\rightarrow
\pi _{t}^{\ast }$ weakly in $\mathcal{P}\left( \Gamma ^{d}\times S\right) $
or equivalently $\Pi _{t}^{\gamma }\overset{\mathcal{D}}{\rightarrow }\pi
_{t}^{\ast }$ in distribution for fixed time $t<T,$ then we have%
\begin{equation}
\Pi _{t}^{\gamma }\overset{\mathbf{P}}{\rightarrow }\pi _{t}^{\ast }\text{
in probability.}  \label{eq:appen-prob}
\end{equation}%
Since $\Pi _{\cdot }^{\gamma }\rightarrow \pi _{\cdot }^{\ast }$ in Skorohod
topology implies $\Pi _{t}^{\gamma }\rightarrow \pi _{t}^{\ast }$ weakly for
continuity points of $\pi _{\cdot }^{\ast }$ \citep[p.112][]{Billingsley68},
it is enough to show that $\pi _{\cdot }^{\ast }:t\mapsto \pi _{t}^{\ast }$
is continuous for all $t\in \lbrack 0,T]$ to obtain (\ref{eq:appen-prob}).
Let $t_{0\text{ }}<T$ and $\left\{ g_{k}\right\} $ a dense family in $%
C(\Gamma \times S)$. Since 
\begin{equation*}
\left\vert \int_{t_{0}}^{t}\left\langle A(\pi _{s}^{\ast
}),g_{k}\right\rangle ds\right\vert \leq (t-t_{0})\sup_{s\in \lbrack
0,T]}\left\langle \mathcal{A}(\pi _{s}^{\ast }),g_{k}\right\rangle
\end{equation*}%
we choose $\delta \,\leq \min \left\{ 1,\epsilon \right\} .$ Then for $%
\left\vert t-t_{0}\right\vert \leq \delta ,$ 
\begin{equation*}
\frac{\left\vert \int_{t_{0}}^{t}\left\langle A(\pi _{s}^{\ast
}),g_{k}\right\rangle ds\right\vert }{1+\left\vert
\int_{t_{0}}^{t}\left\langle A(\pi _{s}^{\ast }),g_{k}\right\rangle
ds\right\vert }\leq \frac{\delta \sup_{s\in \lbrack 0,T]}\left\langle 
\mathcal{A}(\pi _{s}^{\ast }),g_{k}\right\rangle }{1+\delta \sup_{s\in
\lbrack 0,T]}\left\langle \mathcal{A}(\pi _{s}^{\ast }),g_{k}\right\rangle }%
\leq \delta \frac{\sup_{s\in \lbrack 0,T]}\left\langle \mathcal{A}(\pi
_{s}^{\ast }),g_{k}\right\rangle }{1+\sup_{s\in \lbrack 0,T]}\left\langle 
\mathcal{A}(\pi _{s}^{\ast }),g_{k}\right\rangle }\leq \delta
\end{equation*}%
so $\left\Vert \pi _{t}-\pi _{t_{0}}\right\Vert _{\mathcal{P}(\Gamma \times
S)}\,\leq \epsilon $ and $\pi _{\cdot }^{\ast }:t\mapsto \pi _{t}^{\ast }$
is continuous, all $t\in \lbrack 0,T],$ thus all $t\in \lbrack 0,T]$ are
continuity point of $\pi _{\cdot }^{\ast }$
\end{proof}

\noindent \underline{\textbf{Proof of Theorem \ref{thm-long-range-fixedBC}}} 
%\newline

\noindent From Lemma \ref{lem-conv} we have, for $t<T$%
\begin{equation*}
\Pi _{t}^{_{\Lambda \gamma }}(du,di)=\frac{\left\vert \Gamma ^{\gamma
}\right\vert }{\left\vert \Lambda ^{\gamma }\right\vert }\Pi _{t}^{\Gamma
^{\gamma }}(du\cap A,di)\overset{\mathbf{P}}{\rightarrow }\frac{1}{%
\left\vert \Lambda \right\vert }f_{t}(du\cap A,di)
\end{equation*}%
Since $f_{t}(u,i)(du\cap \Lambda )di=f_{\Lambda ,t}(u,i)dudi,$ from (\ref%
{proof-rem2}) we obtain%
\begin{equation*}
\left\langle f_{t},g\right\rangle =\left\langle f_{0},g\right\rangle
+\int_{0}^{t}ds\sum_{k\in S}\int_{\Lambda \times S}c(u,i,k,\frac{1}{%
\left\vert \Gamma \right\vert }f_{s}dudi)\left( g(u,k)-g(u,i)\right)
f_{\Lambda ,s}(dudi)
\end{equation*}%
Since $\left\vert \Gamma ^{\gamma }\right\vert \Pi _{t}^{\Gamma ^{\gamma
}}=\left\vert \Lambda ^{\gamma }\right\vert \Pi _{t}^{\Lambda ^{\gamma
}}+\left\vert \Lambda ^{\gamma }{}^{c}\right\vert \Pi _{0}^{\Lambda ^{\gamma
}{}^{c}},$ $\left\vert \Gamma \right\vert \Pi _{t}^{\Gamma ^{\gamma }}%
\overset{\mathbf{P}}{\rightarrow }f_{t}(u,i)dudi,$ $\left\vert \Lambda
\right\vert \Pi _{t}^{\Lambda ^{\gamma }}\overset{\mathbf{P}}{\rightarrow }%
f_{\Lambda ,t}(u,i)dudi,$ and $\left\vert \Lambda ^{c}\right\vert \Pi
_{0}^{\Lambda ^{\gamma }{}^{c}}\overset{\mathbf{P}}{\rightarrow }f_{\Lambda
^{c}}(u,i)dudi,$ we have $f_{t}=f_{\Lambda ,t}+f_{\Lambda ^{c}}$ for all
\thinspace $t$.

\subsection{Proof of Theorem \protect\ref{Th-Mean-Markov copy(1)}}

To do this first we define a reduction mapping, $\phi :S^{\Lambda
^{n}}\rightarrow \Delta ^{n},$%
\begin{equation*}
\sigma \mapsto \phi (\sigma ),\text{ \ }\phi (\sigma )(i):=\frac{1}{%
\left\vert \Lambda ^{n}\right\vert }\sum\limits_{y\in \Lambda ^{n}}\delta
_{\sigma (y)}(\{i\})
\end{equation*}%
For $g\in L^{\infty }(\Delta ^{n};\mathbb{R})$ we let $f:=g\circ \phi \,$ $%
\in L^{\infty }(S^{\Lambda ^{n}};\mathbb{R})$, where$~f(\sigma )=g(\eta ).$
Then for $\eta =\phi (\sigma )$, we have $f(\sigma ^{x,k})-f(\sigma )=g(\eta
^{\sigma (x),k})-g(\eta )$ since%
\begin{equation*}
\phi (\sigma ^{x,k})(i)=\frac{1}{n^{d}}\sum_{y\in \Lambda _{n}}\delta
_{\sigma (y)}(\{i\})+\frac{1}{n^{d}}\delta _{k}(\{i\})-\frac{1}{n^{d}}\delta
_{\sigma (x)}(\{i\})=\eta ^{\sigma (x),k}(i)
\end{equation*}

\begin{proof}[Proof of Theorm \protect\ref{Th-Mean-Markov copy(1)}]
We check the case of imitative and comparing rates. Other cases can be
treated as a special case. By writing $m^{n}(k):=\sum_{l}a(k,l)\eta ^{n}(l),$
we find%
\begin{eqnarray*}
L_{n}f(\sigma ) &=&\sum_{k\in S}\sum_{x\in \Lambda _{n}}\eta
(k)F(m^{n}(k)-m^{n}(\sigma (x)))(g(\eta ^{\sigma (x),k})-g(\eta )) \\
&=&\sum_{k\in S}\sum_{j\in S}\left[ \sum_{x\in \Lambda _{n}}\delta _{\sigma
(x)}(\{j\})\right] \eta (k)F(m^{n}(k)-m^{n}(j))(g(\eta ^{j,k})-g(\eta )) \\
&=&\sum_{k\in S}\sum_{j\in S}n^{d}\eta (j)\eta
(k)F(m^{n}(k)-m^{n}(j))(g(\eta ^{j,k})-g(\eta )):=\sum_{k\in S}\sum_{j\in
S}n^{d}c^{M}(\eta ,j,k)(g(\eta ^{j,k})-g(\eta ))
\end{eqnarray*}%
Thus we obtain 
\begin{equation*}
L^{n}g(\eta )=\sum_{k\in S}\sum_{j\in S}n^{d}c^{M}(\eta ,j,k)(g(\eta
^{j,k})-g(\eta ))
\end{equation*}%
and this makes $\left\{ \eta _{t}\right\} $ a Markov chain and the rate is
given by $c^{M}(\eta ,j,k).$
\end{proof}

\subsection{Proof of Corollary \protect\ref{thm-uniform}}

\begin{proof}
It is enough to prove the exponential estimate. From (\ref{proof-rem2}) we
recall that 
\begin{equation}
\left\langle \Pi _{t}^{^{\gamma }},g\right\rangle =\left\langle \Pi
_{0}^{^{\gamma }},g\right\rangle +\int_{0}^{t}\sum_{k\in S}\int_{\mathbf{T}%
^{d}\times S}c(u,i,k,\Pi _{s}^{^{\gamma }})\left( g(u,k)-g(u,i)\right) \Pi
_{s}^{^{\gamma }}(dudi)ds+M_{t}^{g,\gamma }  \notag
\end{equation}%
for $g\in C(\mathbf{T}^{d}\times S).$ By taking $g(u,i)=1$ if $i=l$, $%
g(u,i)=0$ otherwise, we find 
\begin{equation*}
\eta _{t,l}^{n}=\eta _{0,l}^{n}+n^{d}\int_{0}^{t}\left[ \sum_{i\in
S}c^{M}(i,l,\eta _{s}^{n})\eta _{s,l}^{n}-\sum_{k\in S}c^{M}(l,k,\eta
_{s}^{n})\eta _{s,l}^{n}\right] ds+M_{t}^{l,n}
\end{equation*}%
We define $\beta _{l}(x):=\sum_{i\in S}\mathbf{c}^{M}(i,l,x)x_{l}-\sum_{k\in
S}\mathbf{c}^{M}(l,k,x)x_{l}$. Thus we have%
\begin{equation*}
\eta _{t,l}^{n}=\eta _{0,l}^{n}+n^{d}\int_{0}^{t}\beta _{l}(\eta
_{s}^{n})ds+M_{t}^{l,n},\text{ \ }\rho _{t,l}=\rho _{0,l}+\int_{0}^{t}\beta
_{l}(\rho _{s})ds
\end{equation*}%
From Lemma \ref{lem-exp-est}, we have $\mathbf{P}\left\{ \sup_{t\leq
T}\left\vert M_{t}^{l,n}\right\vert \geq \delta \right\} \leq 2e^{-n^{d}%
\frac{\delta ^{2}}{TC_{0}}}$ for each $l$ and for $\delta \leq \delta _{0},$
where we note that the choices of $C_{0}$ and $\delta _{0}$ does not depend
on $g$ since $\left\vert g(u,i)\right\vert \leq 1$ for all $u,i.$ Thus, $%
\mathbf{P}\left\{ \sup_{t\leq T}\left\Vert M_{t}^{n}\right\Vert _{u}\geq
\delta \right\} \leq 2\left\vert S\right\vert e^{-\frac{n^{d}\delta ^{2}}{%
TC_{0}}}.$ Therefore for $t\leq T,$ using the Lipschitz condition of $\beta $
we obtain 
\begin{equation*}
\sup_{\tau \leq t}\left\Vert \eta _{\tau }^{n}-\rho _{\tau }\right\Vert
_{u}\leq \left\Vert \eta _{0}^{n}-\rho _{0}\right\Vert
_{u}+L\int_{0}^{t}\sup_{\tau \leq s}\left\Vert \eta _{\tau }^{n}-\rho _{\tau
}\right\Vert _{u}ds+\sup_{t\leq T}\left\Vert M_{t}^{n}\right\Vert _{u}
\end{equation*}%
For $\epsilon _{0}$ in Lemma \ref{lem-exp-est}, we let $\delta =\frac{1}{3}%
e^{-LT}\epsilon $ for $\epsilon <\epsilon _{0}$ and define%
\begin{equation*}
\Omega _{0}=\left\{ \omega :\left\Vert \eta _{0}^{n}-\rho _{0}\right\Vert
_{u}\leq \delta \right\} ,\text{ }\Omega _{1}=\left\{ \omega :\sup_{t\leq
T}\left\Vert M_{t}^{n}\right\Vert _{u}\leq \delta \right\}
\end{equation*}%
Then when $\omega \in \Omega _{0}\cap \Omega _{1},$ we have $\sup_{\tau \leq
T}\left\Vert \eta _{\tau }^{n}-\rho _{\tau }\right\Vert _{u}\leq $ $2\delta
e^{LT}$ by Gronwell lemma. Choose $n_{0\text{ }}$such that $\left\Vert \eta
_{0}^{n}-\rho _{0}\right\Vert _{u}\leq \delta $ for $a.e$. $\omega $ for $%
n\geq n_{0}$. Then for $\epsilon \leq \epsilon _{0}$ and $n\geq n_{0},$%
\begin{eqnarray*}
P\left\{ \sup_{\tau \leq T}\left\Vert \eta _{\tau }^{n}-\rho _{n}\right\Vert
\geq \epsilon \right\} &\leq &P\left( \Omega _{0}^{c}\right) +P\left( \Omega
_{1}^{c}\right) \leq P\left\{ \omega :\left\Vert \eta _{0}^{n}-\rho
_{0}\right\Vert _{u}\geq \delta \right\} +P\left\{ \omega :\sup_{t\leq
T}\left\Vert M_{t}^{n}\right\Vert _{u}\geq \delta \right\} \\
&\leq &2\left\vert S\right\vert e^{-\frac{n^{d}\delta ^{2}}{TC_{0}}%
}=2\left\vert S\right\vert e^{-\frac{n^{d}\epsilon ^{2}}{TC}}
\end{eqnarray*}%
where $C:=9C_{0}e^{2LT}.$
\end{proof}

\subsection{Solutions of Linear IDE}

Applying Fourier transform to (\ref{var-eq}) element by element, we obtain 
\begin{equation}
\frac{\partial \hat{D}(k)}{\partial t}=(M\mathcal{\hat{J}(}k)+N)\hat{D}(k)
\label{IDE-ODE}
\end{equation}%
for each $k\in \mathbb{Z}^{d}$ and $\hat{D}(k)\in \mathbb{C}^{\left\vert
S\right\vert }.$ By solving the ODE system (\ref{IDE-ODE}) for each $k$ and
using the inverse formula, we obtain%
\begin{equation*}
D(x,t)=\sum_{k\in \mathbb{Z}}e^{(M\mathcal{\hat{J}(}k)+N)t}\hat{g}(k)e^{2\pi
ix\cdot k}
\end{equation*}%
where $e^{(M\mathcal{\hat{J}(}k)+N)t}$ is $\ \left\vert S\right\vert \times
\left\vert S\right\vert $ matrix, $\hat{g}(k)$ is $\left\vert S\right\vert
\times 1$ vector.

\subsection{Proof of Proposition \protect\ref{prop-linear-log}}

\begin{proof}
First we note that $p_{1}>\zeta ,$ $p_{2},p_{3}<\zeta $ , $\lim_{\kappa
\rightarrow \infty }p_{2}=\zeta ,$ 
\begin{equation*}
\beta (1-l(\beta (p_{i}-\zeta )))l(\beta (p_{i}-\zeta ))<1\text{ for }i=1,3,%
\text{ }\beta (1-l(\beta (p_{i}-\zeta )))l(\beta (p_{i}-\zeta ))>1\text{ for 
}i=2.
\end{equation*}%
Suppose that $\beta >\beta _{C}$ and consider $p_{1}.$ Since $l(\beta
(p_{1}-\zeta ))=p_{1},$ we have $\beta (1-p_{1})p_{1}<1.$ Then since $%
\mathcal{\hat{J}(}k)\leq 1$ for all $k,$ we have 
\begin{equation*}
\lambda (k)=\beta (1-p_{1})p_{1}\mathcal{\hat{J}(}k)-1<\beta
(1-p_{1})p_{1}-1<0
\end{equation*}%
Thus $p_{1}$ is linearly stable. Similar argument shows that $p_{3}$ is
linearly stable.
\end{proof}

\newpage 
\bibliographystyle{econometrica}
\bibliography{spatial_game}

\end{document}